\documentclass[11pt,reqno]{amsart}

\usepackage{amsmath, amsthm, amssymb}
\usepackage{amsfonts}
\usepackage[ansinew]{inputenc}
\usepackage[dvips]{epsfig}
\usepackage{graphicx}
\usepackage[english]{babel}
\usepackage{comment}
\usepackage{graphicx}
\usepackage{hyperref}
\usepackage{bbm} 

\usepackage{cite}
\usepackage{graphicx}
\usepackage{amscd}
\usepackage{xcolor}
\usepackage{bm}
\usepackage{enumerate}

\usepackage{verbatim}
\usepackage{hyperref}
\usepackage{amstext}
\usepackage{latexsym}
\let\oldsqrt\sqrt
\def\sqrt{\mathpalette\DHLhksqrt}
\def\DHLhksqrt#1#2{%
\setbox0=\hbox{$#1\oldsqrt{#2\,}$}\dimen0=\ht0
\advance\dimen0-0.2\ht0
\setbox2=\hbox{\vrule height\ht0 depth -\dimen0}%
{\box0\lower0.4pt\box2}}

\allowdisplaybreaks

\newcommand{\R}{\mathbb{R}} 
\newcommand{\N}{\mathbb{N}} 

\newcommand{\dist}{\textnormal{dist}} 
\newcommand{\diam}{\textnormal{diam}} 
\newcommand{\supp}{\textnormal{supp}} 
\newcommand{\G}{\Gamma} 

\newcommand{\ov}{\overline}

\renewcommand{\k}{\mathbf{k}}

\renewcommand{\phi}{\varphi}

\newcommand{\cB}{{\mathcal B}}

\newcommand{\cH}{{\mathcal H}}

\newcommand{\cO}{{\mathcal O}}

\newcommand{\cQ}{{\mathcal Q}}

\renewcommand{\O}{\Omega}

\newcommand{\Ds}{(-\Delta)^s}

\theoremstyle{definition}
\newtheorem{defi}{Definition}[section]
\newtheorem{remark}[defi]{Remark}

\theoremstyle{plain} 
\newtheorem{thm}[defi]{Theorem}
\newtheorem{prop}[defi]{Proposition}
\newtheorem{lemma}[defi]{Lemma}

\newcommand{\be}{\begin{equation}}
\newcommand{\ee}{\end{equation}}

\renewcommand{\d}{\delta }

\newcommand{\e }{\varepsilon }
\newcommand{\g }{\gamma}

\renewcommand{\l }{\lambda }

\newcommand{\n }{\nabla }

\newcommand{\calL }{\mathcal{L}}

\newcommand{\calR}{{\mathcal R}}

\theoremstyle{definition}

\numberwithin{equation}{section}

 \title[existence results for regional fractional laplacian]{existence results for nonlocal problems governed by the regional fractional Laplacian}

\author[Mouhamed Moustapha Fall and Remi Yvant Temgoua]{Mouhamed Moustapha Fall$^1$, Remi Yvant Temgoua$^{1,2}$} 

\address{$^1$ African Institute for Mathematical Sciences in Senegal (AIMS Senegal), KM 2, Route de Joal, B.P. 1418. Mbour, S\'en\'egal.}

\email{mouhamed.m.fall@aims-senegal.org}

\address{$^2$ Goethe-Universit\"{a}t Frankfurt, Institut f\"{u}r Mathematik.
Robert-Mayer-Str. 10, D-60629 Frankfurt, Germany.} 

\email{temgoua@math.uni-frankfurt.de, remi.y.temgoua@aims-senegal.org}

\date{\today}

\begin{document}
\maketitle

\begin{abstract}
The aim of the present paper is to study existence results of minimizers  of the critical fractional Sobolev constant on bounded domains.
Under some values of the fractional parameter we show that the best constant is achieved.   If  moreover the underlying domain is a ball, we obtain positive radial minimizers for all possible values of the fractional parameter in higher dimension, while we impose a positive mass condition in low dimension.
\end{abstract}
{\footnotesize
\begin{center}
\end{center}
}

\section{introduction and main results}\label{section:introduction}
Let $\Omega$ be a Lipschitz open set of $\R^N$,  $s\in (1/2,1)$ and $N>2s$.
The purpose of this paper is to study the existence of minimizers to the best Sobolev critical  constant
\begin{equation}\label{minimization-problem-on-domain-0}
S_{N,s}(\Omega)=\inf_{\substack{u\in H^s_0(\Omega)\\ u\neq0}}\frac{Q_{N,s,\Omega}(u)}{\|u\|^2_{L^{2^*_s}(\Omega)}},
\end{equation}
where $H^s_0(\Omega)$ is the completion of $C^{\infty}_c(\Omega)$ with respect to the $H^s(\Omega)$-norm,  $2^*_s:=\frac{2N}{N-2s}$ is the so-called fractional critical Sobolev exponent and 
$Q_{N,s,\Omega}(\cdot)$ is a nonnegative quadratic form defined on $H^s_0(\Omega)$ by
\begin{equation*}
Q_{N,s,\Omega}(u):=\frac{c_{N,s}}{2}\int_{\Omega}\int_{\Omega}\frac{(u(x)-u(y))^2}{|x-y|^{N+2s}}\ dxdy.
\end{equation*} 
We notice that for $s\in(0,1/2]$ and $\O$ bounded, the constant function $1$ belongs to $H^s_0(\Omega)$, and thus, the above Sobolev constant is zero in this case. We refer the reader  to   Appendix \ref{appendix} below for more details and the definition of Lipschitz domains in this paper. 

We recall that nonnegative minimizers of the constant $S_{N,s}(\Omega)$ are weak solutions to nonlinear Dirichlet problem 
\begin{equation}\label{critical-problem-for-regional-0}
\left\{\begin{aligned}
(-\Delta)^s_{\Omega}u&=u^{2^*_s-1}\quad\text{in}\quad\Omega\\
u&=0\quad\quad\quad\text{on}\quad\partial\Omega,
\end{aligned}
\right.
\end{equation} 
where $(-\Delta)^s_{\Omega}$ is the \textit{regional fractional Laplacian} defined as
\begin{equation*}
(-\Delta)^s_{\Omega}u(x)=c_{N,s}P.V.\int_{\Omega}\frac{u(x)-u(y)}{|x-y|^{N+2s}}\ dy,~~x\in\Omega.
\end{equation*}
Here, $c_{N,s}$ is the usual positive normalization constant of $\Ds$ and $P.V.$ stands for the principal value of the integral. 

In the theory of partial differential equations, the existence of solutions of nonlinear equations appears as a natural question. This strongly depends on the type of nonlinearities that are considered. For instance, nonlinear equations involving subcritical power nonlinearities, say $f(t)=|t|^{p-1}$ with $p<2^*_s$, are quite well-understood and due to compactness, the existence of solutions can be easily established by using for example the Mountain Pass theorem. One can also study the corresponding minimization problem and prove that a minimizer exists. Besides, at the critical exponent $p=2^*_s$ we lose compactness and therefore standard argument of calculus of variation cannot be applied to derive the existence of solutions. As a typical example, when $\Omega$ is a star-shaped bounded  domain, it has been proved that the Dirichlet problem
\begin{equation}\label{e}
(-\Delta)^su=u^{2^*_s-1},\quad\quad u>0\quad\text{in}~~\Omega,\ \ \ \ \ u=0\quad\text{in}~~\R^N\setminus\Omega
\end{equation}
does not admit a  solution. Such a nonexistrence result was first proved in \cite{ fall2012nonexistence} and later in  \cite{ros2014pohozaev,ros2017pohozaev}  by means of   a fractional  Pohozaev type identity. However, \eqref{critical-problem-for-regional-0} can have a solution even if $\O$ is star-shaped and smooth.  It is therefore interesting to understand the type of domains and exponents for which \eqref{critical-problem-for-regional-0} does not admit a   solution.

In the case where $\O=\R^N$ or $\O=\R^N_+$, the infinimum  $S_{N,s}(\O)>0$ for all $s\in (0,1)$. Moreover,  see  e.g. \cite{lieb2002sharp} all minimizers of   $S_{N,s}(\R^N)$      are of the form 
\begin{equation}
u(x)=a\Big(\frac{1}{b^2+|x-x_0|^2}\Big)^{\frac{N-2s}{2}},\quad\quad x\in\R^N
\end{equation}
where $a, b$ are positive constants and $x_0\in\R^N$.

Problem of type \eqref{critical-problem-for-regional-0} is less understood in contrast with \eqref{e}. The only paper investigating it is \cite{frank2018minimizers}. Precisely, the authors in \cite{frank2018minimizers} considered the equivalent minimization problem and obtain existence of minimizers under some assumptions on $\O$ and the range of the parameter $s$. In particular,  it is proved in \cite{frank2018minimizers}  that if  a portion of  $\partial\O$ lies on a hyperplane and $N\geq 4s$,  then $S_{N,s}(\Omega)$ is achieved.\\ 
Our first main result removes this assumption on $\O$ provided $s$ is close to $1/2$.
\begin{thm}\label{existence-result-for-s-close-to-1/2-0}
	Let $N\geq2$ and $\Omega\subset\R^N$ be a bounded $C^1$ open set. Then there exists $s_0\in(1/2,1)$ such that for all $s\in (1/2,s_0)$,  the infimum  $S_{N,s}(\Omega)$  is achieved.
\end{thm}
The main ingredient to prove \eqref{existence-result-for-s-close-to-1/2-0} is to  show that  $S_{N,s}(\Omega)<S_{N,s}(\R^N_+)$ for $s$ closed to $1/2$. We achieve this  by showing that $S_{N,1/2}(\Omega)=0$ provided $\O$ is a bounded Lipschitz open set. We notice here that our notion of Lipschitz open set is that $\partial\O$ is locally given by the restriction of a bi-Lipschitz map. This is strictly weaker than the \textit{strongly} Lipschitz property,  meaning  that $\partial \O$ is locally given by  a graph of a Lipschitz function, see Definiton \ref{def:Lipschitz} and Remark \ref{rem:Lipschitz} below.\\

Next, let $\cB$ denote the unit centered ball in $\R^N$.  We consider the minimization problem \eqref{minimization-problem-on-domain-0} on the space $H^s_{0,rad}(\cB)$, the completion of the space of radial functions belonging to  $C^\infty_c(\cB)$ with respect to the norm  $H^s_0(\cB)$.   More precisely,   we consider the infinimum problem, for $h\in L^\infty(\cB)$ being radial,
\begin{equation}\label{minimization-problem-on-domain-0-rad-h}
S_{N,s,rad}(\cB,h)=\inf_{\substack{u\in H^s_{0,rad}(\cB)\\ u\neq0}}\frac{Q_{N,s,\cB}(u)+\int_{\cB}hu^2 dx}{\|u\|^2_{L^{2^*_s}(\cB)}}.
%
\end{equation}
Our next result is related to the existence of minimizers for the infimum $S_{N,s,rad}(\cB,0)$ in high dimension $N\geq4s$.   Our second main result is the following.
\begin{thm}\label{existence-of-radial-minimizers}
	Let $s\in(1/2,1)$ and $N\geq4s$. Then the   infinimum
 \begin{equation}\label{minimization-problem-on-doma-rad}
S_{N,s,rad}(\cB,0)=\inf_{\substack{u\in H^s_{0,rad}(\cB)\\ u\neq0}}\frac{Q_{N,s,\cB}(u) }{\|u\|^2_{L^{2^*_s}(\cB)}}
\end{equation}
  is achieved by a positive function $u\in H^s_{0,rad}(\cB)$, satisfying
  $$
\begin{aligned}
(-\Delta)^s_{\cB}u&=u^{2^*_s-1}\quad\text{in}\quad\cB,\qquad
u=0\quad \text{on}\quad\partial\cB.
\end{aligned}
$$
\end{thm}
We now turn our attention to the minimization problem $S_{N,s,rad}(\cB,h)$ in low dimension $N<4s$. This Sobolev constant is related to the Schr\"{o}dinger operator $(-\Delta)^s_{\cB}+h$. As a necessary condition for the existence of positive minimizers, it is important to assume that $(-\Delta)^s_{\cB}+h$ defines a coercive bilinear form on $H^s_{0,rad}(\cB)$.

Before stated our third main result, we need to introduce the mass of $\cB$ at $0$ associated to the Schr\"{o}dinger operator $(-\Delta)^s+h$, where $(-\Delta)^s$ is the standard fractional Laplacian. Indeed, let $G(x,y)$ be the Green function of the operator $\Ds+h$ on $\cB$ and $\calR$ be the Riesz potential of $\Ds$ on $\R^N$. Then the function  $x\mapsto \k(x)=G(x,0)-\calR(x)$ is continuous in $\cB$. The \textit{mass} of the operator $\Ds+h$ at 0 is given by $\k(0)$.  Our next result is a "positive mass theorem" in the spirit of \cite{ghoussoub2017hardy,schoen1984conformal}.
\begin{thm}\label{existence-of-radial-minimizers-ld}
	Let $s\in(1/2,1)$, $2\leq N<4s$, $h\in L^\infty_{rad}(\cB)$ and suppose that $S_{N,s,rad}(\cB,h)>0$. Assume that $\k(0)>0$.
	 Then  $S_{N,s,rad}(\cB,h)$  is achieved  by a positive function $u\in H^s_{0,rad}(\cB)$, satisfying
  $$
 \begin{aligned}
(-\Delta)^s_{\cB}u+ h u&=u^{2^*_s-1}\quad\text{in}\quad\cB,\qquad
u=0\quad \text{on}\quad\partial\cB.
\end{aligned}
$$
\end{thm}

The role of the mass in proving the existence of minimizers (for Sobolev constant) in low dimensions is very crucial. As we will see later, it helps us to restore the compactness. Indeed, the strict positivity $\k(0)>0$ implies that the Sobolev constant in $\cB$ is strictly less than that of $\R^N$, and thereby produces the existence of  minimizers.

An interesting question that arises is  whether symmetry breaking occurs? More generally, for $p\geq 1$,  is every positive solution to  $u\in H^s_0(\cB)$ to 
$$
\begin{aligned}
(-\Delta)^s_{\cB}u&=u^{p}\quad\text{in}\quad\cB, \qquad u=0  \quad\text{on}\quad\partial\cB,
\end{aligned}
$$
is radial? We conjecture that that the answer to this question is no.\\

In Proposition \ref{boundedness-in-critical-problem} we obtain a priori $L^{\infty}$-bounds of minimizers. Hence, by the ineterior regularity theory and standard  boostrap arguments, they belong to $C^\infty(\O)$, provided $h\in C^\infty(\O)$. In addition, the  boundary regularity result in  \cite{chen2018dirichlet,fall2020regional} implies that minimizers are actually $C^{2s-1}(\ov\Omega)$. \\
 
The rest of the paper is organized as follows. in Section \ref{section:preliminary} we give some preliminaries that will be useful throughout this paper. In Section \ref{section:the critical problem for s close to 1/2} we prove Theorems \ref{existence-result-for-s-close-to-1/2-0}  whereas in Section \ref{section:radial minimizers} we establish Theorems \ref{existence-of-radial-minimizers} and \ref{existence-of-radial-minimizers-ld}. Finally in the Appendix \ref{appendix} we prove that the constant function $1$ belongs to $H^s_0(\Omega)$ for $s\in(0,1/2]$. \\

\textbf{Acknowledgements:} Support from DAAD and BMBF (Germany) within project 57385104 is acknowledged. The first author is also supported by the Alexander von Humboldt Foundation. The authors would like also to thank Tobias Weth and Sven Jarohs for useful discussions.

\section{preliminary}\label{section:preliminary}
In this section, we introduce some preliminary properties which will be useful in this work. For all $s\in(0,1)$, the fractional Sobolev space $H^s(\Omega)$ is defined as the set of all measurable functions $u$ such that
\begin{equation*}
[u]^2_{H^s(\Omega)}:=\frac{c_{N,s}}{2}\int_{\Omega}\int_{\Omega}\frac{(u(x)-u(y))^2}{|x-y|^{N+2s}}\ dxdy
\end{equation*}
is finite. It is a Hilbert space endowed with the norm
\begin{equation*}
\|u\|^2_{H^s(\Omega)}=\|u\|^2_{L^2(\Omega)}+[u]^2_{H^s(\Omega)}.
\end{equation*}
We refer to \cite{di2012hitchhiker's} for more details on this fractional Sobolev spaces. Next, we denote by $H^s_0(\Omega)$ the completion of $C^{\infty}_c(\Omega)$ under the norm $\|\cdot\|_{H^s(\Omega)}$. Moreover, for $s\in(1/2,1)$, $H^s_0(\Omega)$ is a Hilbert space equipped with the norm
\begin{equation*}
\|u\|^2_{H^s_0(\Omega)}=\frac{c_{N,s}}{2}\int_{\Omega}\int_{\Omega}\frac{(u(x)-u(y))^2}{|x-y|^{N+2s}}\ dxdy
\end{equation*}
 which is equivalent to the usual one in $H^s(\Omega)$ thanks to Poincar\'{e} inequality. We define the Hilbert space
\begin{equation*}
\cH^s_0(\Omega)=\{u\in H^s(\R^N):u=0~\text{in}~\R^N\setminus\Omega\}
\end{equation*}
endowed with the norm $\|\cdot\|_{H^s(\R^N)}$, which is the completion of $C^{\infty}_c(\Omega)$ with respect to the norm $\|\cdot\|_{H^s(\R^N)}$. In the sequel, $H^s_{0,rad}(\Omega)$ and $\cH^s_{0,rad}(\Omega)$ are respectively the space of radially symmetric functions of $H^s_0(\Omega)$ and $\cH^s_0(\Omega)$.  

Given $x\in\Omega$ and $r>0$, we denote by $B_r(x)$ the open ball centered at $x$ with radius $r$. When the center is not specified, we will understand that it's the origin, e.g. $B_2(0)=B_2$. The upper half-ball centered at $x$ with radius $r$ is denoted by $B^+_r(x)$. We will always use $\delta_{\Omega}(x)=\dist(x,\partial\Omega)$ for the distance from $x$ to the boundary. For every set $A\subset\R^N$, we denote by  $\mathbbm{1}_{A}$ its characteristic function.
\begin{prop}[see \cite{di2012hitchhiker's,del2015first}]\label{sobolev-embedding}
	The embedding $H^s_0(\Omega)\hookrightarrow L^p(\Omega)$ is continuous for any $p\in[2,2^*_s]$, and compact for any $p\in[2,2^*_s)$.
\end{prop}
The next proposition gives an elementary result regarding the role of convex functions applied to $(-\Delta)^s_{\Omega}$. 
\begin{prop}\label{convex-lemma}
	Assume that $\phi:\R\rightarrow\R$ is a Lipschitz convex function such that $\phi(0)=0$. Then if $u\in H^s_0(\Omega)$ we have
	\begin{equation}
	(-\Delta)^s_{\Omega}\phi(u)\leq\phi'(u)(-\Delta)^s_{\Omega}u\quad\text{weakly in}\quad\Omega.
	\end{equation}
\end{prop}
\begin{proof}
	The proof of the above lemma is standard. In fact, using that every convex $\phi$ satisfies $\phi(a)-\phi(b)\leq\phi'(a)(a-b)$ for all $a,b\in\R$, the proof follows.
\end{proof}
 
We conclude this section showing in proposition below, the boundedness of any nonnegative solution of \eqref{critical-problem-for-regional-0}. The argument uses Moser's iteration method. A similar result has been established in \cite{barrios2015critical} for the case of fractional Laplacian.
\begin{prop}\label{boundedness-in-critical-problem}
	Let $u\in H^s_0(\Omega)$ be a nonnegative solution to problem \eqref{critical-problem-for-regional-0}. Then $u\in L^{\infty}(\Omega)$.
\end{prop}
\begin{proof}
	For $\beta\geq1$ and $T>0$ large, we define the following convex function
	\begin{equation*}
	\phi_{T,\beta}(t)=\left\{\begin{aligned}
	&0,\quad\quad\quad\quad\quad\quad\quad\quad\quad\text{if}\quad t\leq0\\
	&t^{\beta},\quad\quad\quad\quad\quad\quad\quad\quad~~\text{if}\quad 0<t<T\\
	&\beta T^{\beta-1}(t-T)+T^{\beta},\quad\text{if}\quad t\geq T.
	\end{aligned}
	\right.
	\end{equation*}
	Throughout the proof, we will use $\phi_{T,\beta}=:\phi$ for the sake of simplicity. Since $\phi$ is Lipschitz, with constant $\Lambda_{\phi}=\beta T^{\beta-1}$, and $\phi(0)=0$, then $\phi(u)\in H^s_0(\Omega)$ and by the convexity of $\phi$, we have, according to Proposition \ref{convex-lemma} that
	\begin{equation}\label{3}
	(-\Delta)^s_{\Omega}\phi(u)\leq\phi'(u)(-\Delta)^s_{\Omega}u.
	\end{equation}
	By Proposition \ref{sobolev-embedding} and inequality \eqref{3} we have that
	\begin{align*}
	\|\phi(u)\|^2_{L^{2^*_s}(\Omega)}&\leq C\|\phi(u)\|^2_{H^s_0(\Omega)}=C\int_{\Omega}\phi(u)(-\Delta)^s_{\Omega}\phi(u)\ dx\\
	&\leq C\int_{\Omega}\phi(u)\phi'(u)(-\Delta)^s_{\Omega}u\ dx\\
	&=C\int_{\Omega}\phi(u)\phi'(u)u^{2^*_s-1}\ dx.
	\end{align*}
	Moreover, since $u\phi'(u)\leq\beta\phi(u)$, we have that
	\begin{align}\label{4}
	\|\phi(u)\|^2_{L^{2^*_s}(\Omega)}\leq C\beta\int_{\Omega}(\phi(u))^2u^{2^*_s-2}\ dx.
	\end{align}
	We point out that the integral on the right-hand side of the above inequality is finite. Indeed, using that $\beta\geq1$ and $\phi(u)$ is linear when $u\geq T$, we have from a quick computation that
	\begin{align*}
	\int_{\Omega}(\phi(u))^2u^{2^*_s-2}\ dx&=\int_{\{u\leq T\}}(\phi(u))^2u^{2^*_s-2}\ dx+\int_{\{u>T\}}(\phi(u))^2u^{2^*_s-2}\ dx\\&\leq T^{2\beta-2}\int_{\Omega}u^{2^*_s}\ dx+C\int_{\Omega}u^{2^*_s}\ dx<\infty.
	\end{align*}
	 We now choose $\beta$ in \eqref{4} so that $2\beta-1 = 2^*_s$. Denoting by $\beta_1$ such a value, then we can equivalently write
	\begin{equation}\label{6}
	\beta_1:=\frac{2^*_s+1}{2}.
	\end{equation}
	Let $K>0$ be a positive number whose value will be fixed later on. Then applying H$\ddot{\text{o}}$lder's inequality with exponents $q:=2^*_s/2$ and $q':=2^*_s/(2^*_s-2)$ in the integral on the right-hand side of inequality \eqref{4}, we find that
	\begin{align}\label{6'}
\nonumber	&\int_{\Omega}(\phi(u))^2u^{2^*_s-2}\ dx=\int_{\{u\leq K\}}(\phi(u))^2u^{2^*_s-2}\ dx+\int_{\{u> K\}}(\phi(u))^2u^{2^*_s-2}\ dx\\&\leq\int_{\{u\leq K\}}\frac{(\phi(u))^2}{u}K^{2^*_s-1}\ dx+\Bigg(\int_{\Omega}(\phi(u))^{2^*_s}\ dx\Bigg)^{2/2^*_s}\Bigg(\int_{\{u>K\}}u^{2^*_s}\ dx\Bigg)^{\frac{2^*_s-2}{2^*_s}}.
	\end{align}
	Now, thanks to Monotone Convergence Theorem, we can choose $K$ as big as we wish so that
	\begin{equation}\label{7}
	\Bigg(\int_{\{u>K\}}u^{2^*_s}\ dx\Bigg)^{\frac{2^*_s-2}{2^*_s}}\leq\frac{1}{2C\beta_1},
	\end{equation}
	where $C$ is the positive constant appearing in \eqref{4}. Therefore, by taking into account \eqref{7} in \eqref{6'} and by using also \eqref{6}, we deduce from \eqref{4} that
	\begin{equation*}
	\|\phi(u)\|^2_{L^{2^*_s}(\Omega)}\leq2C\beta_1\Bigg(K^{2^*_s-1}\int_{\Omega}\frac{(\phi(u))^2}{u}\ dx\Bigg).
	\end{equation*}
	 Since $\phi(u)\leq u^{\beta_1}$ and recalling \eqref{6}, and by letting $T\rightarrow\infty$, we get that
	\begin{equation*}
	\Bigg(\int_{\Omega}u^{2^*_s\beta_1}\ dx\Bigg)^{2/2^*_s}\leq2C\beta_1\Bigg(K^{2^*_s-1}\int_{\Omega}u^{2^*_s}\ dx\Bigg)<\infty,
	\end{equation*}
	and therefore
	\begin{equation}\label{8}
	u\in L^{2^*_s\beta_1}(\Omega).
	\end{equation}
	Suppose now that $\beta>\beta_1$. Thus, using that $\phi(u)\leq u^{\beta}$ in the right
	hand side of \eqref{4} and letting $T\rightarrow\infty$ we get
	\begin{equation}\label{9}
	\Bigg(\int_{\Omega}u^{2^*_s\beta}\ dx\Bigg)^{2/2^*_s}\leq C\beta\Bigg(\int_{\Omega}u^{2\beta+2^*_s-2}\ dx\Bigg).
	\end{equation}
	Therefore,
	\begin{equation}\label{10}
	\Bigg(\int_{\Omega}u^{2^*_s\beta}\ dx\Bigg)^{\frac{1}{2^*_s(\beta-1)}}\leq(C\beta)^{\frac{1}{2(\beta-1)}}\Bigg(\int_{\Omega}u^{2\beta+2^*_s-2}\ dx\Bigg)^{\frac{1}{2(\beta-1)}}.
	\end{equation}
	 We are now in position to use an iterative argument as in \cite[Proposition 2.2]{barrios2015critical}. For that, we define inductively the sequence $\beta_{m+1},~m\geq1$ by
	\begin{equation*}
	2\beta_{m+1}+2^*_s-2=2^*_s\beta_m,
	\end{equation*}
	from which we deduce that,
	\begin{equation*}
	\beta_{m+1}-1=\Big(\frac{2^*_s}{2}\Big)^m(\beta_1-1).
	\end{equation*}
	Now by using $\beta_{m+1}$ in place of $\beta$, in \eqref{10}, it follows that
	\begin{align*}
	\Bigg(\int_{\Omega}u^{2^*_s\beta_{m+1}}\ dx\Bigg)^{\frac{1}{2^*_s(\beta_{m+1}-1)}}\leq(C\beta_{m+1})^{\frac{1}{2(\beta_{m+1}-1)}}\Bigg(\int_{\Omega}u^{2^*_s\beta_m}\ dx\Bigg)^{\frac{1}{2^*_s(\beta_m-1)}}.
	\end{align*}
	For the sake of clarity, we set
	\begin{equation*}
	C_{m+1}:=(C\beta_{m+1})^{\frac{1}{2(\beta_{m+1}-1)}}\quad\text{and}\quad A_m:=\Bigg(\int_{\Omega}u^{2^*_s\beta_m}\ dx\Bigg)^{\frac{1}{2^*_s(\beta_m-1)}}
	\end{equation*}
	so that
	\begin{equation}
	A_{m+1}\leq C_{m+1}A_m,~~m\geq1.
	\end{equation}
	Then iterating the above inequality, we find that
	\begin{equation*}
	A_{m+1}\leq\prod^{m+1}_{i=2}C_iA_1,
	\end{equation*}
	which implies that
	\begin{align*}
	\log A_{m+1}&\leq\sum_{i=2}^{m+1}\log C_i+\log A_1\\
	&\leq\sum_{i=2}^{\infty}\log C_i+\log A_1.
	\end{align*}
	Since $\beta_{m+1}=(\beta_1-1/2)^m(\beta_1-1)+1$ then the serie $\sum_{i=2}^{\infty}\log C_i$ converges. Also, since $u\in L^{2^*_s\beta_1}(\Omega)$ (see \eqref{8}), then $A_1\leq C$. From this, we find that
	\begin{equation}
		\log A_{m+1}\leq C_0
	\end{equation}
	with being $C_0>0$ a positive constant independent of $m$. By letting $m\rightarrow\infty$, it follows that
	\begin{equation*}
	\|u\|_{L^{\infty}(\Omega)}\leq C_0'<\infty.
	\end{equation*}
 This completes the proof.
\end{proof}

\section{Existence of minimizers for $s$ close to $1/2$}\label{section:the critical problem for s close to 1/2}
 We aim to study the existence of nontrivial solutions of \eqref{critical-problem-for-regional-0}. As pointed point out in the introduction the embedding $H^s_0(\Omega)\hookrightarrow L^{2^*_s}(\Omega)$ fails to be compact and due to this, the functional energy associated to \eqref{critical-problem-for-regional-0} does not satisfy the Palais-Smale compactness condition. Hence finding the critical points by standard variational methods become a very tough task.  Therefore, a natural question arises:
 
  \[\textit{\textbf{$(\cQ)$} Does problem \eqref{critical-problem-for-regional-0} admits a nontrivial solution?}\]\\
  In other words, we are looking at whether the quantity
\begin{equation}\label{minimization-problem-on-domain}
S_{N,s}(\Omega)=\inf_{\substack{u\in H^s_0(\Omega)\\ u\neq0}}\frac{Q_{N,s,\Omega}(u)}{\|u\|^2_{L^{2^*_s}(\Omega)}}
\end{equation}
is attained or not. Here $Q_{N,s,\Omega}(\cdot)$ is a nonnegative quadratic form define on $H^s_0(\Omega)$ by
\begin{equation*}
Q_{N,s,\Omega}(u):=\frac{c_{N,s}}{2}\int_{\Omega}\int_{\Omega}\frac{(u(x)-u(y))^2}{|x-y|^{N+2s}}\ dxdy.
\end{equation*}
As a quick comment on the above question, Frank et al. \cite[Theorem 4]{frank2018minimizers} gave a positive answer in the special case of a class of $C^1$ open sets whose boundary has a flat part, that is $C^1$ domains $\Omega$ with the shape $B^+_r(z)\subset\Omega\subset\R^N_+$ for some $r>0$ and $z\in\partial\R^N_+$, and such that $\R^N_+\setminus\Omega$ has nonempty interior. This flatness assumption on the boundary of $\Omega$ allows the authors in \cite{frank2018minimizers} to obtain the strict inequality $S_{N,s}(\Omega)<S_{N,s}(\R^N_+)$, which is the crucial ingredient for the proof of Theorem 4 in there. Notice that in \cite{frank2018minimizers}, the question remains open for a larger class of sets.

In the sequel, we give a positive affirmation to the above question in the case of arbitrary open sets with $C^1$ boundary, provided that $s$ is close to $1/2$. As a consequence, one has in contrast with the fractional Laplacian that   the above question has a positive answer even if $\O$ is convex and of class $C^\infty$.\\

 For the reader's convenience, we restate our main result in the following.
\begin{thm}\label{existence-result-for-s-close-to-1/2}
	Let $N\geq2$ and $\Omega\subset\R^N$ be a bounded Lipschitz open set. There exists $s_0\in(1/2,1)$ such that for all $s\in(1/2,s_0)$, any minimizing sequence for $S_{N,s}(\Omega)$, normalized in $H^s_0(\Omega)$ is relatively compact in $H^s_0(\Omega)$. In particular, the infimum is achieved.
\end{thm}
The proof of the above main theorem is a direct consequence of the key proposition below (see Proposition \ref{asymptotic-behavior-of-crritical-level-as-s-tends-to-1/2}), in which we examine the asymptotic behavior of the Sobolev critical constant $S_{N,s}(\Omega)$ as $s$ tends to $1/2^{+}$, by showing that the latter goes to zero. The proof of this only requires the domain to be Lipschitz. 
Our key proposition is stated as follows.
\begin{prop}\label{asymptotic-behavior-of-crritical-level-as-s-tends-to-1/2}
Let $\Omega\subset\R^N$ be a bounded Lipschitz open set. Then 
	\begin{equation}\label{limit-of-the-critical-value}
	\lim\limits_{s\searrow1/2}S_{N,s}(\Omega)=0.
	\end{equation}
\end{prop}
We now collect some interesting results that are needed to complete the proof of Proposition \ref{asymptotic-behavior-of-crritical-level-as-s-tends-to-1/2} above.  Let us start with the following upper semicontinuous lemma.
\begin{lemma}\label{continuous-lemma}
	Let $\Omega\subset\R^N$ be a bounded Lipschitz open set. Fix $s_0\in[1/2,1)$. Then
	\begin{equation}
	\limsup_{s\searrow s_0}S_{N,s}(\Omega)\leq S_{N,s_0}(\Omega).
	\end{equation}
\end{lemma}
\begin{proof}
	For $t\in\R$, we recall the elementary inequality
	\begin{equation}\label{elementary-identity-1}
	|e^t-1|\leq\sum_{k=1}^{+\infty}\frac{|t|^k}{k!}\leq\sum_{k=1}^{+\infty}\frac{|t|^k}{(k-1)!}\leq|t|e^{|t|}.
	\end{equation}
	For all $r,\gamma>0$, we also recall the following growth regarding the logarithmic function:
	\begin{equation}\label{logarithmic-decays}
	|\log|z||\leq\frac{1}{e\gamma}|z|^{-\gamma}~~\text{if}~~ |z|\leq r~~~~\text{and}~~~~|\log|z||\leq\frac{1}{e\gamma}|z|^{\gamma}~~ \text{if}~~|z|\geq r.
	\end{equation}
	Let $\epsilon>0$ and let $u_{\epsilon}\in C^{\infty}_c(\Omega)$ such that $\|u_{\epsilon}\|_{L^{2^*_s}(\Omega)}=1$ and $Q_{N,s_0,\Omega}(u_{\epsilon})\leq S_{N,s_0}(\Omega)+\epsilon$. Then $S_{N,s}(\Omega)\leq Q_{N,s,\Omega}(u_{\epsilon})$. From this, we obtain that
	\begin{align}\label{difference-of-critical-values}
	S_{N,s}(\Omega)-S_{N,s_0}(\Omega)\leq  Q_{N,s,\Omega}(u_{\epsilon})- Q_{N,s_0,\Omega}(u_{\epsilon})+\epsilon.
	\end{align}
	On the other hand,
	\begin{align*}
	&|Q_{N,s,\Omega}(u_{\epsilon})- Q_{N,s_0,\Omega}(u_{\epsilon})|\\
	&\leq\frac{1}{2}|c_{N,s}-c_{N,s_0}|\int_{\Omega}\int_{\Omega}\frac{(u_{\epsilon}(x)-u_{\epsilon}(y))^2}{|x-y|^{N+2s_0}}\ dxdy\\
	&~~~~~~~~~~~~+\frac{c_{N,s}}{2}\int_{\Omega}\int_{\Omega}\frac{(u_{\epsilon}(x)-u_{\epsilon}(y))^2}{|x-y|^{N+2s_0}}||x-y|^{2(s_0-s)}-1|\ dxdy\\
	&\leq\frac{1}{c_{N,s_0}}(S_{N,s_0}(\Omega)+\epsilon)|c_{N,s}-c_{N,s_0}|\\
	&~~~~~~~~~~~~+\frac{c_{N,s}}{2}\int_{\Omega}\int_{\Omega}\frac{(u_{\epsilon}(x)-u_{\epsilon}(y))^2}{|x-y|^{N+2s_0}}||x-y|^{2(s_0-s)}-1|\ dxdy.
	\end{align*}
	Next, from \eqref{elementary-identity-1} we have that
	\begin{align*}
	||x-y|^{2(s_0-s)}-1|=|e^{2(s_0-s)\log|x-y|}-1|&\leq2|s_0-s||\log|x-y||e^{2|s_0-s||\log|x-y||}\\&=2|s_0-s||\log|x-y|||x-y|^{2|s_0-s|}.
	\end{align*}
	Taking this into account and using the regularity of $u_{\epsilon}$ and the property \eqref{logarithmic-decays}, we find that
	\begin{align}\label{bounds-of-the-difference-of-quadratics-forms}
\nonumber&|Q_{N,s,\Omega}(u_{\epsilon})- Q_{N,s_0,\Omega}(u_{\epsilon})|\\
	&\leq\frac{1}{c_{N,s_0}}(S_{N,s_0}(\Omega)+\epsilon)|c_{N,s}-c_{N,s_0}|+Cc_{N,s}\diam(\Omega)^{2|s_0-s|}|s_0-s|+\epsilon
	\end{align}
	where $\diam(\Omega)=\sup\{|x-y|:x,y\in\Omega\}$ is the diameter of $\Omega$ and $C=C(N,s_0,\gamma,\Omega,u_{\epsilon})>0$ is a positive constant. Now, by letting $s\searrow s_0$ in \eqref{bounds-of-the-difference-of-quadratics-forms} we obtain that
	\begin{equation*}
	\limsup_{s\searrow s_0}|Q_{N,s,\Omega}(u_{\epsilon})- Q_{N,s_0,\Omega}(u_{\epsilon})|\leq\epsilon.
	\end{equation*}
	Since $\epsilon$ can be chosen arbitrarily small, it follows that
	\begin{equation*}
	\limsup_{s\searrow s_0}|Q_{N,s,\Omega}(u_{\epsilon})- Q_{N,s_0,\Omega}(u_{\epsilon})|=0
	\end{equation*}
	and therefore, we deduce from \eqref{difference-of-critical-values} that
	\begin{equation}\label{limsup-of-the-critical-level}
	\limsup_{s\searrow s_0}S_{N,s}(\Omega)\leq S_{N,s_0}(\Omega),
	\end{equation}
	as desired.
\end{proof}
We have the following proposition. Its proof is given in the Appendix \ref{appendix}.
\begin{prop}\label{an-interresting-approximation-result-for-the-half-critical-value}
	Let $\Omega$ be a bounded Lipschitz open set of $\R^N$. Then
	\begin{equation}
	S_{N,1/2}(\Omega)=0.
	\end{equation}
\end{prop} 
We can now give the proof of our key proposition.
\begin{proof}[Proof of Proposition \ref{asymptotic-behavior-of-crritical-level-as-s-tends-to-1/2}]
	Since $S_{N,s}(\Omega)>0$ then if follows that
	\begin{equation}\label{liminf-critical-value}
	\liminf_{s\searrow1/2} S_{N,s}(\Omega)\geq0.
	\end{equation}
	On the other hand, applying Lemma \ref{continuous-lemma} together with Proposition \ref{an-interresting-approximation-result-for-the-half-critical-value}, we have that
	\begin{equation}\label{limsup-critical-value}	\limsup_{s\searrow1/2}S_{N,s}(\Omega)\leq S_{N,1/2}(\Omega)=0,
	\end{equation}
	Now, from \eqref{liminf-critical-value} and \eqref{limsup-critical-value} we deduce \eqref{limit-of-the-critical-value}, and this ends the proof of Proposition \ref{asymptotic-behavior-of-crritical-level-as-s-tends-to-1/2}.
\end{proof}
Having the above key tools in mind, we can now give the proof of Theorem \ref{existence-result-for-s-close-to-1/2}.
\begin{proof}[Proof of Theorem \ref{existence-result-for-s-close-to-1/2}]
	Let $s\in(1/2,1)$ with $s$ close to $1/2$. Then by Proposition \ref{asymptotic-behavior-of-crritical-level-as-s-tends-to-1/2}, we have that $S_{N,s}(\Omega)\rightarrow0$ as $s\searrow1/2$. Consequently, for $s$ close to $1/2$, and since $S_{N,s}(\R^N_+)>0$ for all $s\in (0,1)$ (see e.g. \cite[Lemma 2.1]{dyda2011fractional}), we deduce that
	\begin{equation}\label{key-condition-for-the-existence-for-critical-problem}
	0<S_{N,s}(\Omega)<S_{N,s}(\R^N_+)\quad\text{for all}\quad s\in(1/2,s_0)
	\end{equation}
for some $s_0\in (1/2,1)$. With the above key inequality, we complete the proof by following closely the argument developed by Frank et al. \cite{frank2018minimizers} for the proof of Theorem 4 in there.
\end{proof}
\begin{remark}
	Since $Q_{N,s,\Omega}(|u|)\leq Q_{N,s,\Omega}(u)$ then the minimizer in \eqref{minimization-problem-on-domain}, or equivalently, the solution of \eqref{critical-problem-for-regional-0} can be assumed nonnegative.
\end{remark}

\section{The radial problem}
In the present section, we consider the existence of minimizers to   quotient
\be\label{eq:defSNs}
S_{N,s,rad}(\cB,h):=\inf_{u\in C^\infty_{c,rad}(\cB)} \frac{[u]_{H^s(\cB)}^2+\int_{\cB} hu^2dx}{\|u\|_{L^{2^*_s}(\cB)}^2}.
\ee
Here and in the following, we consider the class of radial potentials $h\in L^\infty(\cB)$ such that 
\be \label{eq:defSNsOh}
S_{N,s,rad}(\cB,h)>0.
\ee
 We observe that if $h(x)\equiv -\l$ with $\l<\l_1(\cB)$, the first eigenvalue of $\Ds_{\cB}$, then \eqref{eq:defSNsOh} holds.
%
The aim of this section is to provide situations in which $S_{N,s,rad}(\cB,h)<S_{N,s}(\R^N) .$
%
\begin{remark}
We  observe that if $ h$ satisfies \eqref{eq:defSNsOh}, then  if $u\in H^s_0(\cB)$ satisfies, weakly, $\Ds_{\cB} u=f$ in $\cB$ with $f \in L^p(\cB)$, for some $p>\frac{N}{2s}$, then $u\in C(\cB)\cap L^\infty(\cB)$. This follows from the argument of Proposition \ref{boundedness-in-critical-problem} and the interior regularity.
\end{remark}
We start recalling the following result from  \cite{frank2018minimizers}.
\begin{prop}\label{key-prop-1} [\cite[Proposition 7]{frank2018minimizers}]
	Let $s\in (1/2,1)$ and $N\geq4s$. Then
	\begin{equation}\label{key-assumption-for-radial-minimization-on-the-unit-ball}
	S_{N,s,rad}(\cB, 0)<S_{N,s}(\R^N) .
	\end{equation}
\end{prop}

The following result plays a crucial role for the existence theorems.
\begin{prop}\label{key-prop-3}
	Let $1/2<s<1$ and $N\geq2$. Then there is a constant $C=C(N,s)>0$ such that for all $u\in H^s_{0,rad}(\cB)$,
	\begin{equation}
	Q_{N,s,\cB}(u)\geq S_{N,s}(\R^N)\|u\|^2_{L^{2^*_s}(\cB)}-C_{\cB}\|u\|^2_{L^2(\cB)}.
	\end{equation}
\end{prop}

  For this, we need the following two lemmas.
\begin{lemma}\label{l1}
	For every $\rho\in(0,1)$, there exists $K_{\rho}>0$ with the property that
	\begin{equation*}
	Q_{N,s,\cB}(u)\geq S_{N,s}(\R^N)\|u\|^2_{L^{2^*_s}(\cB)}-K_{\rho}\|u\|^2_{L^2(\cB)}\quad\text{for every}~u\in H^s_{0,rad}(\cB)~\text{with}~\supp u\subset B_{\rho}.
	\end{equation*}
\end{lemma}
\begin{proof}
	Let $u\in H^s_{0,rad}(\cB)$ with $\supp u\subset B_{\rho}$. We have
	\begin{equation*}
	Q_{N,s,\cB}(u)=Q_{N,s,\R^N}(u)-\int_{\cB}\kappa_{\cB}(x)u(x)^2\ dx\geq S_{N,s}(\R^N)\|u\|^2_{L^{2^*_s}(\cB)}-\int_{\cB}\kappa_{\cB}(x)u(x)^2\ dx,
	\end{equation*}
	with being $\kappa_{\cB}$ the killing measure for $\cB$ define as $\kappa_{\cB}(x)=c_{N,s}\int_{\R^N\setminus\cB}\frac{1}{|x-y|^{N+2s}}\ dy,~x\in\cB$.
  On the other hand, since $\supp u\subset B_{\rho}$, then
	\begin{align*}
	\int_{\cB}\kappa_{\cB}(x)u(x)^2\ dx=	\int_{B_{\rho}}\kappa_{\cB}(x)u(x)^2\ dx
	\end{align*}
	and for every $x\in B_\rho$,
	\begin{equation*}
	\kappa_{\cB}(x)=c_{N,s}\int_{\R^N\setminus\cB}\frac{dy}{|x-y|^{N+2s}}\leq c_{N,s}\int_{|z|\geq 1-\rho}|z|^{-N-2s}\ dz=a_{N,s}(1-\rho)^{-2s}.
	\end{equation*}
	Taking this into account, we find that
	\begin{align*}
	\int_{\cB}\kappa_{\cB}(x)u(x)^2\ dx\leq a_{N,s}(1-\rho)^{-2s}\int_{B_{\rho}}u(x)^2\ dx\leq K_{\rho}\|u\|^2_{L^2(B_\rho)}\leq K_{\rho}\|u\|^2_{L^2(\cB)},
	\end{align*}
	with $K_{\rho}=a_{N,s}(1-\rho)^{-2s}$. From this, we get that
	\begin{equation*}
	Q_{N,s,\cB}(u)\geq S_{N,s}(\R^N)\|u\|^2_{L^{2^*_s}(\cB)}-K_{\rho}\|u\|^2_{L^2(\cB)},
	\end{equation*}
	concluding the proof.
\end{proof}
\begin{lemma}\label{l2}
	For every $M, \rho>0$ there exists $C_{\rho,M}>0$ with
	\begin{equation*}
	Q_{N,s,\cB}(u)\geq M\|u\|^2_{L^{2^*_s}(\cB)}-C_{\rho,M}\|u\|^2_{L^2(\cB)}\quad\text{for every}~u\in H^s_{0,rad}(\cB)~\text{with}~u\equiv0~\text{in}~B_\rho.
	\end{equation*}
\end{lemma}
\begin{proof}
	We first recall that for $s\in (1/2,1)$, $H^s_0(\cB)=\cH^s_0(\cB)$. Therefore, for every $u\in H^s_{0,rad}(\cB)\subset H^s_0(\cB)=\cH^s_0(\cB)$, we have $u\in\cH^s_{0,rad}(\cB)$. Thus, combining the fractional version of the Strauss radial lemma (see \cite[Lemma 2.5]{dinh2019existence}) and the Hardy inequality (see \cite{dyda2004fractional}) we get that
	\begin{align}
\nonumber	|u(x)|^2&\leq \gamma_{N,s}|x|^{-(N-2s)}Q_{N,s,\R^N}(u)=\gamma_{N,s}|x|^{-(N-2s)}\Bigg(Q_{N,s,\cB}(u)+\int_{\cB}\kappa_{\cB}(x)u(x)^2\ dx\Bigg)\\& \nonumber\leq \gamma_{N,s}|x|^{-(N-2s)}\Bigg(Q_{N,s,\cB}(u)+\gamma_{N,s,\cB}\int_{\cB}\delta_{\cB}(x)^{-2s}u(x)^2\ dx\Bigg)\\&\leq
d_{N,s,\cB}|x|^{-(N-2s)}Q_{N,s,\cB}(u),
	\end{align}
	which implies that
	\begin{equation}
	\|u\|^2_{L^{\infty}(\cB\setminus B_{\rho})}\leq d_{N,s,\cB}\rho^{-(N-2s)}Q_{N,s,\cB}(u)\quad\text{for every}~u\in H^s_{0,rad}(\cB)~\text{with}~u\equiv0~\text{in}~B_\rho.
	\end{equation}
	Consequently, using interpolation and Young's inequality with exponents $p=2/\alpha$ and $p'=2/(2-\alpha)$, we find that, for all $M>0$,
	\begin{align*}
	\|u\|^2_{L^{2^*_s}(\cB\setminus B_{\rho})}&\leq C\|u\|^{\alpha}_{L^2(\cB\setminus B_{\rho})}\|u\|^{2-\alpha}_{L^{\infty}(\cB\setminus B_{\rho})}\\
	&\leq\frac{1}{Md_{N,s,\cB}\rho^{-(N-2s)}}\|u\|^2_{L^{\infty}(\cB\setminus B_{\rho})}+\frac{C_{\rho,M}}{M}\|u\|^2_{L^{2}(\cB\setminus B_{\rho})}
	\end{align*}
	with suitable constants $\alpha\in(0,2)$ and $C_{\rho,M}>0$, and hence
	\begin{align*}
	M\|u\|^2_{L^{2^*_s}(\cB\setminus B_{\rho})}&\leq\frac{1}{d_{N,s,\cB}\rho^{-(N-2s)}}\|u\|^2_{L^{\infty}(\cB\setminus B_{\rho})}+C_{\rho,M}\|u\|^2_{L^{2}(\cB\setminus B_{\rho})}\\
	&\leq Q_{N,s,\cB}(u)+C_{\rho,M}\|u\|^2_{L^2(\cB)}
	\end{align*}
	for every $u\in H^s_{0,rad}(\cB)$ with $u\equiv0$ in $B_\rho$. The claim follows.
\end{proof}
In the following, we give the
\begin{proof}[Proof of Proposition \ref{key-prop-3}]
	We choose $0<\rho_2<\rho_1<1$. Moreover, let $\chi_1, \chi_2\in C^{\infty}_c(\R^N)$ with $0\leq\chi_i\leq1$, $\chi^2_1+\chi^2_2\equiv1$ in $\cB$ and $\supp\chi_1\subset B_{\rho_1}$, $\supp\chi_2\subset\R^N\setminus\overline{B_{\rho_2}}$. Then we can write $u=\chi^2_1u+\chi^2_2u$ in $\cB$. \\
 Applying $Q_{N,s,\cB}(\cdot)$ to $u=\sum_{i=1}^{2}\chi^2_iu$, we easily find that
	\begin{equation}\label{r1}
	Q_{N,s,\cB}(u)=\sum_{i=1}^{2}Q_{N,s,\cB}(\chi_iu)-\frac{c_{N,s}}{2}\sum_{i=1}^{2}\int_{\cB}\int_{\cB}\frac{(\chi_i(x)-\chi_i(y))^2}{|x-y|^{N+2s}}u(x)u(y)\ dxdy.
	\end{equation}
By the regularity of $\chi_i$, we observe that there is no singularity in the double integral and therefore it follows from the Schur test that there exists a positive constant $C>0$ such that
\begin{equation}\label{r2}
\sum_{i=1}^{2}\int_{\cB}\int_{\cB}\frac{(\chi_i(x)-\chi_i(y))^2}{|x-y|^{N+2s}}u(x)u(y)\ dxdy\leq C\int_{\cB}u^2\ dx.
\end{equation}
In fact, we can write
\begin{align}\label{schur-1}
\int_{\cB}\int_{\cB}\frac{(\chi_i(x)-\chi_i(y))^2}{|x-y|^{N+2s}}u(x)u(y)\ dxdy&\leq C\int_{\cB}\int_{\cB}K(x,y)u(x)u(y)\ dxdy\\&=C\int_{\cB}Tu(x)u(x)\ dx
\end{align}
where 
\begin{equation*}
Tu(x)=\int_{\cB}K(x,y)u(y)\ dy\quad\quad\text{with}\quad\quad K(x,y)=|x-y|^{2-N-2s}.
\end{equation*}
Moreover, by H$\ddot{\text{o}}$lder inequality,
\begin{equation}\label{schur-2}
\int_{\cB}Tu(x)u(x)\ dx\leq\|Tu\|_{L^2(\cB)}\|u\|_{L^2(\cB)}.
\end{equation}
Now, the Schur test implies that there is $C>0$ such that
\begin{equation}\label{schur-3}
\|Tu\|_{L^2(\cB)}\leq C\|u\|_{L^2(\cB)}.
\end{equation}
Therefore, inequality \eqref{r2} follows by combining \eqref{schur-1}, \eqref{schur-2} and \eqref{schur-3}.\\
On the other hand, by Lemmas \ref{l1} and \ref{l2}, there exists a positive constant $C>0$, depending on $\rho_1$ and $\rho_2$ with the property that
\begin{equation}\label{r3}
Q_{N,s,\cB}(\chi_iu)\geq S_{N,s}(\R^N)\|\chi_iu\|^2_{L^{2^*_s}(\cB)}-C\|\chi_iu\|^2_{L^2(\cB)}.
\end{equation}
Plugging \eqref{r2} and \eqref{r3} into \eqref{r1}, we find that
\begin{align}
Q_{N,s,\cB}(u)\geq S_{N,s}(\R^N)\sum_{i=1}^{2}\|\chi_iu\|^2_{L^{2^*_s}(\cB)}-C\sum_{i=1}^{2}\|\chi_iu\|^2_{L^2(\cB)}.
\end{align}
Next, since $\sum_{i=1}^{2}\chi^2_i=1$, we have
\begin{align*}
\sum_{i=1}^{2}\|\chi_iu\|^2_{L^{2^*_s}(\cB)}&=\sum_{i=1}^{2}\Big\|\chi^2_iu^2\Big\|_{L^{\frac{N}{N-2s}}(\cB)}\geq\Bigg\|\sum_{i=1}^{2}\chi^2_iu^2\Bigg\|_{L^{\frac{N}{N-2s}}(\cB)}\\&=\|u^2\|_{L^{\frac{N}{N-2s}}(\cB)}=\|u\|^2_{L^{2^*_s}(\cB)}.
\end{align*}
Using this in \eqref{r3}, it follows that
\begin{equation*}
Q_{N,s,\cB}(u)\geq S_{N,s}(\R^N)\|u\|^2_{L^{2^*_s}(\cB)}-C\|u\|^2_{L^2(\cB)},
\end{equation*}
completing the proof.
\end{proof}

\subsection{The case $2s<N<4s$}

We now  let  $G(x,y)$ be the Green function of $\Ds+h$, with zero exterior Dirichlet boundary data. Letting $G(x)=G(x,0)$, we have that  
\be\label{eq:Green}
\begin{cases}
\Ds G(x)+h(x)G(x)=\d_0(x)&\textrm{  in $\cB$}\\
~~~~~~~~~~~~~~~~~~~~~~~~G(x)=0& \textrm{  in $\R^N\setminus \cB$},
\end{cases}
\ee
where $\d_0$ is the Dirac mass at  $0$. We recall that $G$ is a radial function. In fact this follows from the construction and uniqueness of Green function.
We let $\calR(x)=t_{N,s} |x|^{2s-N}$ be the Riesz potential of $\Ds$ on $\R^N$. It satisfies 
\be \label{eq:riesz}
\Ds  \calR(x)=\d_0(x), 
\ee
where  $t_{N,s}:=\pi^{-\frac{N}{2}}2^{-s}\frac{\G((N-s)/2)}{\G(s/2)}$.    
We  now define  $\ov \k\in L^1(\cB)$, by 
\be
\ov \k(x):=G(x)- \calR(x). 
\ee
%
%
It then follows, from  \eqref{eq:Green},  that 
\be\label{eq:cLk}
\Ds\ov  \k(x)+h(x)\ov  \k(x)=-h(x)\calR(x).
\ee
  Since $N<4s$, we have that $\ov \k\in L^2(\cB)$ and  $h\calR\in L^p(\cB)\cap L^2(\cB)$, for some  $p>\frac{N}{2s}$.  Therefore, by regularity theory,  
$
\ov \k\in  C(\ov{\cB})  .
$
Recall that $\ov \k(y)$ is  the \textit{mass} of $\cB$  associated to the operator $\calL_{\R^N}:=\Ds+ h(x)$.    We remark that  if $ \chi\in C^\infty_c(\cB)$, with $\chi=1$ in a neighborhood of $0$, then  letting 
$$
 \k(x):=G(x)- \chi(x)\calR(x),
$$
then, by continuity,  $\k(y)=\ov\k (y)$, for all $y\in \cB$. This follows from the fact that $\Ds  \k+h  \k\in L^p(\cB)$, for some  $p>\frac{N}{2s}$ and thus $ \k\in C(\cB)$.
\begin{remark}
It would be interesting to find  potential  $h$ for which $  \k(0)>0$.
\end{remark}
First, for $\e>0$ we set
	\begin{equation*}
	u_{\e}(x)=\gamma_0\Big(\frac{\e}{\e^2+ |x|^2}\Big)^{\frac{N-2s}{2}},
	\end{equation*}
	where $\gamma_0$ is a positive constant (independent of $\epsilon$) such that $\|u_{\e}\|_{L^{2^*_s}(\R^N)}=1$. It is known that $u_{\e}$ satisfies the Euler-Lagrange equation
	\begin{equation}
	(-\Delta)^su_{\e}=S_{N,s} u^{2^*_s-1}_{\e}\quad\text{in}\quad\R^N.
	\end{equation}
Our next result shows that in low dimension $N<4s$, the positive mass implies existence of minimizers. 
\begin{lemma}\label{lem:strict-ineq-ld}
Suppose that $2s<N<4s$.  Suppose that $ \k(0)>0$. Then 
\be\label{key-inequality-low-dimension}
S_{N,s,rad}(\cB,h)<  S_{N,s} :=S_{N,s}(\R^N) .
\ee
\end{lemma}
\begin{proof}
For $r\in (0,1/4)$,  we let $\eta\in C^\infty_c(B_{2r})$ be radial,  with $\eta=1$ on $B_{r}$.  
We define the test function $v_\e\in H^s_{0,rad}(\cB)$ given by 
\begin{align}
v_\e (x)&=\eta(x) u_\e(x)+\e^{\frac{N-2s}{2}}\frac{\g_0}{t_{N,s}} \left( G(x)- \eta(x)\calR(x) \right) \nonumber \\
&=\eta(x) u_\e(x)+\e^{\frac{N-2s}{2}}\frac{\g_0}{t_{N,s}}  \k(x)  .
\end{align}
%
%
We define $W_\e:=\eta u_\e-\e^{\frac{N-2s}{2}}\frac{\g_0}{t_{N,s}}  \eta\calR$ and $a_s:=\frac{\g_0}{t_{N,s}} $. 
%

Note that  $\e^{-\frac{N-2s}{2}} W_\e\to 0 \in C_{loc}(\R^N\setminus\{0\})\cap L^1(\cB)$ and $|\e^{-\frac{N-2s}{2}} u_\e (x)|\leq \g_0 |x|^{2s-N}$. Hence, since $N<4s$, we deduce that  $|x|^{2(2s-N)}\in L^1_{loc}(\R^N)$ and thus by the dominated convergence theorem,
 \be
 \int_{\cB} u_\e(x) h(x)W_\e(x)\, dx= o(\e^{N-2s}).
\ee
We then have 
\begin{align*}
& [v_\e]^{2}_{H^s(\cB)}+ \int_{\cB}hv_\e^2\, dx\leq   [v_\e]^{2}_{H^s(\R^N)}+ \int_{\cB}hv_\e^2\, dx=   \int_{\cB} v_\e(x) \calL_{\R^N} v_\e(x)\, dx\\
& \leq  \e^{\frac{N-2s}{2}}a_s  \int_{\cB} v_\e(x) \calL_{\R^N}    G(x) \, dx  +  \int_{\cB} v_\e(x) \calL_{\R^N} W_\e(x)\, dx \\
& \leq   \e^{\frac{N-2s}{2}}a_s  u_\e(0)+  \e^{{N-2s}}a_s^2 \k(0) +   \int_{\cB}\eta u_\e(x) \Ds W_\e(x)\, dx\\
&+   \e^{\frac{N-2s}{2}} a_s\int_{\cB}\k(x) \calL_{\R^N} W_\e(x)\, dx+ o(\e^{N-2s})\\
%
%
&\leq   \e^{\frac{N-2s}{2}} a_s  u_\e(0)+ \e^{{N-2s}}a_s^2  \k(0) +   \int_{\cB}\eta u_\e(x) \Ds  (\eta u_\e)(x) \, dx\\
&-  \e^{\frac{N-2s}{2}} a_s \int_{\cB}\eta u_\e(x) \Ds  ( \eta\calR)(x)\, dx +  \e^{\frac{N-2s}{2}} a_s\int_{\cB}\k(x) \calL_{\R^N} W_\e(x)\, dx + o(\e^{N-2s}) \\
&\leq   \e^{\frac{N-2s}{2}} a_s  u_\e(0)+ \e^{{N-2s}}a_s^2  \k(0) \\
&+   \int_{\R^N}\eta u_\e(x) \Ds (\eta u_\e)(x) \, dx -  \e^{\frac{N-2s}{2}} a_s \int_{\R^N}\eta u_\e(x) \Ds   ( \eta\calR)(x)\, dx  \\
&    +  \e^{\frac{N-2s}{2}} a_s\int_{\R^N}\k(x) \calL_{\R^N}  W_\e(x)\, dx+ o(\e^{N-2s}).
\end{align*}
Letting $\ov W_\e=u_\e- \e^{\frac{N-2s}{2}} a_s \calR(x)$, since $N<4s$,  we have that
\be\label{eq:ovWe-conve}
\e^{-\frac{N-2s}{2}} \ov W_\e \to 0\qquad\textrm{ in $C^1_{loc}(\R^N\setminus\{0\})\cap \calL^1_s\cap L^2_{loc}(\R^N)$}.
\ee
Therefore, using that $\Ds\calR=\d_0$ and $\Ds u_\e=S_{N,s} u_\e^{2^*_s-1}$, we get 
 \begin{align*}
& \e^{\frac{N-2s}{2}} a_s  u_\e(0)+  \int_{\R^N}\eta u_\e(x) \Ds (\eta u_\e)(x) \, dx-  \e^{\frac{N-2s}{2}} a_s  \int_{\R^N}\eta u_\e(x) \Ds   ( \eta\calR)(x)\, dx\\
& =\e^{\frac{N-2s}{2}} a_s  u_\e(0) + \int_{\R^N}\eta^2 u_\e(x) \Ds  u_\e(x) \, dx - \e^{\frac{N-2s}{2}} a_s \int_{\R^N}\eta u_\e(x) \Ds  \calR(x) \, dx\\
&+  \int_{\R^N}\eta u_\e(x) \ov W_\e(x) \Ds  \eta(x) \, dx   - \int_{B_{2r}}\eta u_\e(x) J_\e(x)dx\\
&=S_{N,s}\int_{\R^N}\eta^2 u_\e^{2^*_s}+  \int_{\R^N}\eta u_\e(x) \ov W_\e(x) \Ds  \eta(x) \, dx   - \int_{B_{2r}}\eta u_\e(x)J_\e(x) dx\\
&=S_{N,s}\int_{\R^N}\eta^2 u_\e^{2^*_s}+ o( \e^{N-2s}) -\int_{B_{2r}}\eta u_\e(x)  J_\e(x) dx,
 \end{align*}
 where $ J_\e(x):={c_{N,s}} \int_{\R^{N}}\frac{(\overline{W}_\e(x)-\overline{W}_\e(y))(     \eta(x)-\eta(y))}{|x-y|^{N+2s}}\, dy$. To estimate $J_\e$,   we consider first  $x\in B_{r/2}$ and thus 
 \begin{align*}
 J_\e(x)= c_{N,s}\int_{|y|>r}\frac{(\overline{W}_\e(x)-\overline{W}_\e(y))(     \eta(x)-\eta(y))}{|x-y|^{N+2s}}\, dy=o( \e^{\frac{N-2s}{2}}) O( |x|^{\frac{N-2s}{2}}) .
 \end{align*}
 If now $|x|\geq r/2$, we estimate 
 \begin{align*}
| J_\e(x)|&\leq c_{N,s}\int_{|y|<{r/4}}\frac{|(\overline{W}_\e(x)-\overline{W}_\e(y))(     \eta(x)-\eta(y))|}{|x-y|^{N+2s}}\, dy\\
&~~~~~~~~~~~~~~~~~~~~+c_{N,s} \int_{|y|>{r/4}}\frac{|(\overline{W}_\e(x)-\overline{W}_\e(y))(     \eta(x)-\eta(y))|}{|x-y|^{N+2s}}\, dy\\
&\leq  o( \e^{\frac{N-2s}{2}})+ \|\n \eta\|_{L^\infty(\R^N)}   \int_{4r>|y|>{r/4}}\frac{\sup_{t\in [0,1]}|\n \overline{W}_\e(\g_{x,y}(t))| |\g'_{x,y}(t)| }{|x-y|^{N+2s-1}}\, dy\\
&=o( \e^{\frac{N-2s}{2}})
 \end{align*}
 where $\g_{x,y}:[0,1]\to B_{r/2}\setminus B_{r/4}$ is the $C^1$  shortest curve satisfying $\g_{x,y}(0)=x$,  $\g_{x,y}(1)=y$ and $\sup_{t\in [0,1]} |\g'_{x,y}(t)|\leq C |x-y|$. 
Since $N<4s$, by \eqref{eq:cLk} and \eqref{eq:ovWe-conve},  we have  
 \begin{align*}
&  \left|\int_{\R^N}\k(x) \calL_{\R^N}  W_\e(x)\, dx \right|\leq   \left|\int_{B_{2r}}|\calL_{\R^N}\k(x)|  |W_\e(x)|\, dx \right|   =o(  \e^{\frac{N-2s}{2}} ) .
 \end{align*}
 We thus conclude that 
 \begin{align}\label{eq:sem-norm-v-eps}
 [v_\e]^{2}_{H^s(\cB)}+ \int_{\cB} h v_\e^2\, dx&\leq S_{N,s}\int_{\R^N}\eta^2 u_\e^{2^*_s}+  \e^{{N-2s}}a_s^2  \k(0)+o( \e^{{N-2s}})+ O( \e^{{N-2s}})o_r(1) \nonumber\\
 &\leq S_{N,s}+  \e^{{N-2s}}a_s^2  \k(0)+o( \e^{{N-2s}})+ O(r^{4s-N} \e^{{N-2s}}).
 \end{align}
  Since $2^*_s>2$, there exists a positive constant $C(N,s)$  such that
\begin{equation*}
||a+b|^{2^*_s}-|a|^{2^*_s}-2^*_s ab |a|^{2^*_s-2}| \leq C(N,s) \left(|a|^{2^*_s-2} b^2+|b|^{2^*_s}\right)\qquad\textrm{  for all $a,b \in \R$.}
\end{equation*}
As a consequence, with $a=\eta(x) u_\e(x)$ and $b=\e^{\frac{N-2s}{2}} a_s\k(x) $,  we obtain
 \begin{align*}
&\int_{\cB}v_\e^{2^*_s}- \int_{\R^N}(\eta u_\e)^{2^*_s}=2^*_s \e^{\frac{N-2s}{2}} a_s \int_{\cB} (\eta u_\e)^{2^*_s-1}\k(x)\, dx\\
&+ o( \e^{{N-2s}})+ O\left( \e^{N-2s} \int_{\R^N}| \eta(x) u_\e(x)|^{2^*_s-2} \k^2(x)dx \right)\\
&=2^*_s \e^{\frac{N-2s}{2}} \frac{a_s}{S_{N,s}} \int_{\cB} \eta ^{2^*_s-1}\k(x) \Ds u_\e\, dx+ o( \e^{{N-2s}})+\e^{{N-2s}} O\left(\|\eta u_\e\|_{L^{2^*_s}(B_{2r})}^{{2^*_s-2}} \|\k\|_{L^{2^*_s}(B_{2r})}^{{2}} \right). \\
%
%
%
&=2^*_s \e^{\frac{N-2s}{2}} \frac{a_s}{S_{N,s}} \int_{\cB} \k(x) \Ds  \overline{W}_\e\, dx+ 2^*_s \e^{\frac{N-2s}{2}} \frac{a_s}{S_{N,s}} \int_{\cB}(\eta ^{2^*_s-1}-1)\k(x) \Ds  \overline{W}_\e\, dx\\
&+ 2^*_s \e^{{N-2s}} \frac{a_s^2}{S_{N,s}}\k(0)+ o( \e^{{N-2s}})+ O(\e^{{N-2s}} r^{N-2s} ) \\
&=2^*_s \e^{\frac{N-2s}{2}} \frac{a_s}{S_{N,s}} \int_{\cB}  \overline{W}_\e(x) \calL_{\R^N}\k(x)   \, dx+ 2^*_s \e^{\frac{N-2s}{2}} \frac{a_s}{S_{N,s}} \int_{\cB}(\eta ^{2^*_s-1}-1)\k(x) \Ds  \overline{W}_\e\, dx\\
&+ 2^*_s \e^{{N-2s}} \frac{a_s^2}{S_{N,s}}\k(0)+ o( \e^{{N-2s}}) + O(\e^{{N-2s}} r^{N-2s} )\\
&=2^*_s \e^{{N-2s}} \frac{a_s^2}{S_{N,s}}O\left( \int_{|x|<2r} |x|^{2s-N}\left(\frac{1}{(\e^2+|x|^2)^{\frac{N-2s}{2}}} -\frac{1}{|x|^{N-2s}} \right)  \, dx\right)\\
&+ 2^*_s \e^{\frac{N-2s}{2}} \frac{a_s}{S_{N,s}} \int_{\cB}(\eta ^{2^*_s-1}-1)\k(x) \Ds \ov W_\e\, dx + 2^*_s \e^{{N-2s}} \frac{a_s^2}{S_{N,s}}\k(0)+ o( \e^{{N-2s}})+O(\e^{{N-2s}} ) o_r(1).
 \end{align*}
 We estimate
 \begin{align*}
 &\int_{\cB}(\eta ^{2^*_s-1}-1)\k(x) \Ds \ov W_\e\, dx= \int_{\cB}(\eta ^{2^*_s-1}-1)\k(x) \Ds(\eta_{r/4} \ov W_\e)\, dx+o(   \e^{\frac{N-2s}{2}} ) \\ 
 &=c_{N,s}\int_{|x|\geq r}(1-\eta ^{2^*_s-1}(x))\k(x)\int_{ |y|<r/2}\frac{\eta_{r/4} (y)\overline{W}_\e(y)\, dy}{|x-y|^{N+2s}}\, dy +o(   \e^{\frac{N-2s}{2}} )=o(   \e^{\frac{N-2s}{2}} ).
 \end{align*}
 Here, from the definition of $\eta$, we define $\eta_{r/4}\in C^{\infty}_c(B_{r/2})$ with $\eta_{r/4}=1$ on $B_{r/4}$. From the above estimates, we then obtain
\begin{align*}
 \int_{\cB}v_\e^{2^*_s}&=  \int_{\R^N}(\eta u_\e)^{2^*_s}  + 2^*_s \e^{{N-2s}} \frac{a_s^2}{S_{N,s}}\k(0)+ o( \e^{{N-2s}})+O(\e^{{N-2s}} ) o_r(1)\\
 &=1+ 2^*_s \e^{{N-2s}} \frac{a_s^2}{S_{N,s}}\k(0)+ o( \e^{{N-2s}})+O(\e^{{N-2s}} ) o_r(1).
 \end{align*}
 Combining this with \eqref{eq:sem-norm-v-eps}, we finally get
$$
\frac{  [v_\e]_{H^s(\cB)}^2+ \int_{\cB}hv_\e^2\, dx}{\|v_\e\|_{L^{2^*_s}(\cB)}^2}\leq S_{N,s}-  \e^{{N-2s}} {a_s^2} \k(0)+ o( \e^{{N-2s}})+O(\e^{{N-2s}} ) o_r(1).
$$
This finishes the proof.
 \end{proof}

\section{Existence of radial minimizers}\label{section:radial minimizers}
The goal of this section is to investigate the existence of a radial solution of problem \eqref{critical-problem-for-regional-0} in the case when $\Omega=\cB$ is the unit ball of $\R^N$, $N>2s$. More precisely, we aim to analyze the attainability of the following radial critical level
\begin{equation}\label{radial-critical-level-in-the-unit-ball-1}
S_{N,s,rad}(\cB,h)=\inf_{\substack{u\in H^s_{0,rad}(\cB)\\ u\neq0}}\frac{Q_{N,s,\cB}(u)+\int_{\cB}hu^2\, dx}{\|u\|^2_{L^{2^*_s}(\cB)}}.
\end{equation}
To this end, we make use of the method of missing mass as in \cite{frank2018minimizers}. The idea is to prove that a minimizing sequence  for $S_{N,s,rad}(\cB,h)$ does not concentrate at the origin. For that, we will exploit inequalities \eqref{key-assumption-for-radial-minimization-on-the-unit-ball} and \eqref{key-inequality-low-dimension} respectively for high $(N\geq4s)$ and low $(2s<N<4s)$ dimensions.

For the reader's convenience, we restate the main result of this subsection in the following.
\begin{thm}\label{main-result-for-radial-minimization-probem}
	Let $s\in(1/2,1)$,  $N>2s$ and $h\in L^\infty(\cB)$ be a radial function.  Suppose that 
	$0<S_{N,s,rad}(\cB,h)<S_{N,s}(\R^N). $
	Then any minimizing sequence for $S_{N,s,rad}(\cB)$, normalized in $H^s_{0,rad}(\cB)$ is relatively compact in $H^s_{0,rad}(\cB)$ . In particular, the infimum is achieved.
\end{thm}
To prove the above theorem, we first collect some useful results. Let's introduce
\begin{equation}\label{S-start-2}
S^*_{N,s,rad}(\cB):=\inf\Big\{\liminf_{k\rightarrow\infty}\|u_k\|^{-2}_{L^{2^*_s}(\cB)}:Q_{N,s,\cB}(u_k)=1,~u_k\rightharpoonup 0~\text{in}~H^s_{0,rad}(\cB)\Big\}.
\end{equation}
We have the following interesting one-sided inequality.
\begin{prop}\label{key-prop-2}
	Let $1/2<s<1$ and $N\geq2$. Then
	\begin{equation}
	S^*_{N,s,rad}(\cB)\geq S_{N,s}(\R^N).
	\end{equation}
\end{prop}
%
\begin{proof} 
	Let $(u_k)\subset H^s_{0,rad}(\cB)$ with $Q_{N,s,\cB}(u_k)=1$ and $u_k\rightharpoonup0$ in $H^s_{0,rad}(\cB)$. Then by Proposition \ref{key-prop-3} there is $C_{\cB}>0$ such that
	\begin{align*}
	Q_{N,s,\cB}(u_k)\geq S_{N,s}(\R^N)\|u_k\|^2_{L^{2^*_s}(\cB)}-C_{\cB}\|u_k\|^2_{L^2(\cB)}.
	\end{align*}
	By the compact embedding $H^s_{0,rad}(\cB)\hookrightarrow L^2(\cB)$, we have $u_k\rightarrow0$ in $L^2(\cB)$. Using this and by passing to the limit in the above inequality, we find that
	\begin{align*}
	1\geq S_{N,s}(\R^N)\limsup_{k\rightarrow\infty}\|u_k\|^2_{L^{2^*_s}(\cB)},
	\end{align*}
	that is,
	\begin{equation*}
	\liminf_{k\rightarrow\infty}\|u_k\|^{-2}_{L^{2^*_s}(\cB)}\geq S_{N,s}(\R^N).
	\end{equation*}
From the above inequality, we conclude the proof.
\end{proof}

Having collected the above results, we are ready to prove our main result.
\begin{proof}[Proof of Theorem \ref{main-result-for-radial-minimization-probem}]
	Let $(u_k)$ be a minimizing sequence for $S_{N,s,rad}(\cB,h)$, which is normalized in $H^s_{0,rad}(\cB)$. Then after passing to a subsequence, there is $u\in H^s_{0,rad}(\cB)$ such that
	\begin{equation}
	\begin{aligned}
	&u_k\rightharpoonup u\quad\text{weakly in}~~H^s_{0,rad}(\cB)\\
	&u_k\rightarrow u\quad\text{strongly in}~~L^2(\cB)\\
	&u_k\rightarrow u\quad\text{a.e. in}~~\cB.
	\end{aligned}
	\end{equation} 
 Now, by setting $w_k=u_k-u$, it follows that $w_k\rightharpoonup 0$ weakly in $H^s_{0,rad}(\cB)$. Using this, we have that
	\begin{equation}
	1=Q_{N,s,\cB,h}(u_k):=Q_{N,s,\cB}(u_k)+\int_{\cB}hu_k^2\,dx =Q_{N,s,\cB,h}(u)+Q_{N,s,\cB}(w_k)+o(1),
	\end{equation}
	where  $Q_{N,s,\cB,h}(u):=Q_{N,s,\cB}(u)+\int_{\cB}hu^2\,dx $.
	From the above identities, we see that $Q_{N,s,\cB}(w_k)$ converges, say, to $R_1$, which satisfies according to the above equality,
	\begin{equation}\label{equality-r1}
	1=Q_{N,s,\cB,h}(u)+R_1.
	\end{equation}
	Moreover, using that $u_k\rightarrow u$ a.e. in $\cB$ and the Brezis-Lieb lemma \cite{brezis1983relation}, we get that
	\begin{equation}
	S_{N,s,rad}(\cB,h)^{-\frac{N}{N-2s}}+o(1)=\|u_k\|^{\frac{2N}{N-2s}}_{L^{2^*_s}(\cB)}=\|u\|^{\frac{2N}{N-2s}}_{L^{2^*_s}(\cB)}+\|w_k\|^{\frac{2N}{N-2s}}_{L^{2^*_s}(\cB)}+o(1),
	\end{equation}
	from which we deduce that $\int_{\cB}|w_k|^{\frac{2N}{N-2s}}\ dx$ converges, say, to $R_2$ satisfying
	\begin{equation}\label{equality-r2}
	S_{N,s,rad}(\cB,h)^{-\frac{N}{N-2s}}=\|u\|^{\frac{2N}{N-2s}}_{L^{2^*_s}(\cB)}+R_2.
	\end{equation}	
	Now by Proposition \ref{key-prop-2} we easily see that
	\begin{equation}\label{inequality-r3}
	R_1\geq S_{N,s}(\R^N)R_2^{\frac{N-2s}{N}}.
	\end{equation}
The above inequality follows immediately if $R_2=0$. Otherwise, if $R_2>0$, then it suffices to use $\tilde{w}_k:=w_k/Q_{N,s,\cB}(w_k)^{1/2}$ in the definition of $S^*_{N,s,rad}(\cB)$ since $\tilde{w}_k\rightharpoonup0$ weakly in $H^s_{0,rad}(\cB)$ and $Q_{N,s,\cB}(\tilde{w}_k)=1$ as well.\\

From \eqref{equality-r1}, \eqref{equality-r2}, \eqref{inequality-r3} and by using the elementary inequality  \footnote{$0\leq b\leq a\Rightarrow0\leq b/a\leq1$ and then $0\leq b/a\leq(b/a)^{\alpha}\leq1$ for all $0\leq\alpha\leq1$. Hence,
		\begin{equation*}
		\frac{a^{\alpha}-b^{\alpha}}{(a-b)^{\alpha}}=\frac{1-(b/a)^{\alpha}}{(1-(b/a))^{\alpha}}\leq\frac{1-(b/a)}{(1-(b/a))^{\alpha}}\leq1.
		\end{equation*}
	}
	\begin{equation}
	(a-b)^{\alpha}\geq a^{\alpha}-b^{\alpha}\quad\text{for}~0\leq\alpha\leq1,~a\geq b\geq0
	\end{equation}
	with $\alpha=(N-2s)/N$, we find that
	\begin{align*}
	1&=Q_{N,s,\cB,h}(u)+R_1\\
	&\geq Q_{N,s,\cB,h}(u)+S_{N,s}(\R^N)R_2^{\frac{N-2s}{N}}\\
	&=Q_{N,s,\cB,h}(u)+(S_{N,s}(\R^N)-S_{N,s,rad}(\cB,h))R^{\frac{N-2s}{N}}_2\\
	&\ \ \ \ \ \ \ \ \ \ \ +S_{N,s,rad}(\cB)\Big(S_{N,s,rad}(\cB,h)^{-\frac{N}{N-2s}}-\|u\|^{\frac{2N}{N-2s}}_{L^{2^*_s}(\cB)}\Big)^{\frac{N-2s}{N}}\\
	 &\geq Q_{N,s,\cB,h}(u)+(S_{N,s}(\R^N)-S_{N,s,rad}(\cB,h))R^{\frac{N-2s}{N}}_2\\ &\ \ \ \ \ \ \ \ \ \ \ +S_{N,s,rad}(\cB,h)\Big(S_{N,s,rad}(\cB,h)^{-1}-\|u\|^{2}_{L^{2^*_s}(\cB)}\Big)\\
	  &=Q_{N,s,\cB,h}(u)+(S_{N,s}(\R^N)-S_{N,s,rad}(\cB,h))R^{\frac{N-2s}{N}}_2+1-S_{N,s,rad}(\cB,h)\|u\|^{2}_{L^{2^*_s}(\cB)}.
	\end{align*}
	Thus,
	\begin{align}\label{t}
	Q_{N,s,\cB,h}(u)-S_{N,s,rad}(\cB,h)\|u\|^{2}_{L^{2^*_s}(\cB)}+(S_{N,s}(\R^N)-S_{N,s,rad}(\cB,h))R^{\frac{N-2s}{N}}_2\leq0.
	\end{align}
	Since $Q_{N,s,\cB,h}(u)\geq S_{N,s,rad}(\cB,h)\|u\|^{2}_{L^{2^*_s}(\cB)}$ and  $S_{N,s}(\R^N)>S_{N,s,rad}(\cB,h)$ by assumption, it follows from \eqref{t} that $R_2=0$ which implies that $u\not\equiv0$ thanks to \eqref{equality-r2}. Therefore,
	\begin{equation*}
	Q_{N,s,\cB,h}(u)\leq S_{N,s,rad}(\cB,h)\|u\|^2_{L^{2^*_s}(\cB)},
	\end{equation*}
	which implies that $u$ is an optimizer. Therefore, instead of the inequality \eqref{inequality-r3}, we have equality, yielding $R_1=0$. This implies that $Q_{N,s,\cB,h}(u)=1$ and from this, we conclude that $(u_k)$ converges strongly in $H^s_{0,rad}(\cB)$. The proof is therefore finished. 
\end{proof}

\begin{proof}[Proof of Theorem \ref{existence-of-radial-minimizers} and Theorem \ref{existence-of-radial-minimizers-ld} (completed)]
The proof of Theorem \ref{existence-of-radial-minimizers} and Theorem \ref{existence-of-radial-minimizers-ld}  are immediate consequences of Theorem \ref{main-result-for-radial-minimization-probem},  Lemma \ref{lem:strict-ineq-ld} and Proposition  \ref{key-prop-1}.
\end{proof}

\section{appendix}\label{appendix}
In this section, we prove that the constant function $1$ belongs to $H^s_0(\Omega)$ for $s\in(0,1/2]$. By Sobolev embedding, it is enough to treat the case $s=1/2$.

For every $k\in\N$, we define $\chi_k\in C^{0,1}(\R_+)$  by
\begin{equation}\label{test-function-1d}
\chi_k(t)=\left\{\begin{aligned}
&0\quad\quad\quad\quad\quad\text{if}\quad t\leq\frac{1}{k^2}, \\
&\frac{\log k^2t}{|\log1/k|}\quad~~\text{if}\quad \frac{1}{k^2}\leq t\leq\frac{1}{k}, \\
&1\quad\quad\quad\quad\quad\text{if}\quad t\geq\frac{1}{k}.
\end{aligned}
\right.
\end{equation}
We wish now to approximate the constant function $1$ with respect to the $H^{1/2}(\Omega)$-norm. The general strategy is to build an approximation sequence with $\chi_k$ together with a partition of unity. Before going further in our analysis, we need first of all a $one$-dimensional approximation argument.
\begin{lemma}\label{1-d-approximation}
	We have
	\begin{equation}
	\chi_k\rightarrow1\quad\quad\text{in}~H^{1/2}(\R_+)~~\text{as}~~k\rightarrow\infty.
	\end{equation}
\end{lemma}
\begin{proof}
	Clearly, by definition $\chi_k\rightarrow1$ a.e. in $\R_+$. The goal is to show that
	\begin{equation}\label{H-half-convergence}
	\|\chi_k-1\|_{H^{1/2}(\R_+)}\rightarrow0\quad\quad\text{as}~~k\rightarrow\infty.
	\end{equation}
	We start by proving that
	\begin{equation}\label{L2-convergence}
	\|\chi_k-1\|_{L^2(\R_+)}\rightarrow0\quad\quad\text{as}~~k\rightarrow\infty.
	\end{equation}
	We have
	\begin{align*}
	\|\chi_k-1\|^2_{L^2(\R_+)}&=\int_{0}^{\infty}(\chi_k-1)^2\ dt=\int_{0}^{1/k^2}(\chi_k-1)^2\ dt+\int_{1/k^2}^{1/k}(\chi_k-1)^2\ dt\\
	&=\frac{1}{k^2}+\int_{1/k^2}^{1/k}\Big(\frac{\log k^2t}{\log k}-1\Big)^2\ dt=\frac{1}{k^2}+\frac{1}{k^2}\int_{1}^{k}\Big(\frac{\log t}{\log k}-1\Big)^2\ dt\\
	&=\frac{1}{k^2}+\frac{1}{k\log^2k}\int_{1/k}^{1}\log^2t\ dt=\frac{1}{k^2}+\frac{1}{k^2\log^2k}\Big(2-\frac{\log^2k}{k}-\frac{2\log k}{k}-\frac{2}{k}\Big).
	\end{align*}
	From the estimate above, \eqref{L2-convergence} follows.\\
	
	Next, we also prove that
	\begin{equation}
	[\chi_k-1]_{H^{1/2}(\R_+)}\rightarrow0\quad\quad\text{as}~~k\rightarrow\infty.
	\end{equation}
	
	We have
	\begin{align*}
	[\chi_k-&1]^2_{H^{1/2}(\R_+)}=\frac{c_{1,1/2}}{2}\int_{0}^{\infty}\int_{0}^{\infty}\frac{(\chi_k(x)-\chi_k(y))^2}{(x-y)^2}\ dxdy\\
	&~~~~=c\Bigg(\int_{0}^{1/k}\int_{0}^{1/k}\cdots+2\int_{0}^{1/k}\int_{1/k}^{\infty}\cdots+\int_{1/k}^{\infty}\int_{1/k}^{\infty}\cdots\Bigg)\frac{(\chi_k(x)-\chi_k(y))^2}{(x-y)^2}\ dxdy.
	\end{align*}
	Since $\chi_k(x)=\chi_k(y)=1$ for $(x,y)\in (1/k,\infty)\times (1/k,\infty)$ then the third integral in the above equality vanishes. Therefore,
	\begin{equation*}
	[\chi_k-1]^2_{H^{1/2}(\R_+)}=c\int_{0}^{\infty}\int_{0}^{\infty}\frac{(\chi_k(x)-\chi_k(y))^2}{(x-y)^2}\ dxdy=c(I_k+J_k) 
	\end{equation*}
	where
	\begin{equation*}
	I_k:=\int_{0}^{1/k}\int_{0}^{1/k}\frac{(\chi_k(x)-\chi_k(y))^2}{(x-y)^2}\ dxdy\quad\text{and}\quad J_k:=2\int_{0}^{1/k}\int_{1/k}^{\infty}\frac{(\chi_k(x)-\chi_k(y))^2}{(x-y)^2}\ dxdy.
	\end{equation*}\\
	\textbf{Estimate of $J_k$.}
	We have
	\begin{align*}
	&\int_{0}^{1/k}\int_{1/k}^{\infty}\frac{(\chi_k(x)-\chi_k(y))^2}{(x-y)^2}\ dxdy\\
	&=\Bigg(\int_{0}^{1/k^2}\int_{1/k}^{\infty}\cdots+\int_{1/k^2}^{1/k}\int_{1/k}^{\infty}\cdots\Bigg)\frac{(\chi_k(x)-\chi_k(y))^2}{(x-y)^2}\ dxdy\\
	&=J^1_k+J^2_k
	\end{align*}
	where
	\begin{equation*}
	J^1_k:=\int_{0}^{1/k^2}\int_{1/k}^{\infty}\frac{(\chi_k(x)-\chi_k(y))^2}{(x-y)^2}\ dxdy\quad\text{and}\quad J^2_k:=\int_{1/k^2}^{1/k}\int_{1/k}^{\infty}\frac{(\chi_k(x)-\chi_k(y))^2}{(x-y)^2}\ dxdy.
	\end{equation*}
	Regarding $J^1_k$, we have from the definition of $\chi_k$ that
	\begin{align}\label{a-1}
	\nonumber J^1_k&=\int_{0}^{1/k^2}\int_{1/k}^{\infty}\frac{1}{(x-y)^2}\ dxdy\stackrel{\tau=\frac{x}{y}}{=}\int_{0}^{1/k^2}\frac{1}{y}\int_{1/ky}^{\infty}\frac{1}{(\tau-1)^2}\ d\tau dy\\
	&=\int_{0}^{1/k^2}\frac{k}{1-ky}\ dy=-\log\Big(1-\frac{1}{k}\Big).
	\end{align}
	For $J^2_k$, we also use the definition of $\chi_k$ to see that
	\begin{align}\label{a-2}
	\nonumber J^2_k&=\int_{1/k^2}^{1/k}\int_{1/k}^{\infty}\frac{\Big(1-\frac{\log k^2x}{\log k}\Big)^2}{(x-y)^2}\ dxdy=\frac{1}{\log^2k}\int_{1/k^2}^{1/k}\int_{1/k}^{\infty}\frac{(\log k-\log k^2x)^2}{(x-y)}\ dxdy\\
	\nonumber&=\frac{1}{\log^2k}\int_{1/k^2}^{1/k}\int_{1/k}^{\infty}\frac{(\log kx)^2}{(x-y)^2}\ dxdy\stackrel{\substack{\tau=kx\\ t=ky}}{=}\frac{1}{\log^2k}\int_{1/k}^{1}\int_{1}^{\infty}\frac{\log^2\tau}{(\tau-t)^2}\ d\tau dt\\
	\nonumber&=\frac{1}{\log^2k}\int_{1}^{\infty}\Big(\frac{1}{(\tau-\frac{1}{k})}-\frac{1}{(\tau-1)}\Big)\log^2\tau\ d\tau\\
	&=\frac{1}{\log^2k}\int_{1}^{\infty}\frac{\frac{1}{k}-1}{(\tau-\frac{1}{k})(\tau-1)}\log^2\tau\ d\tau.
	\end{align}
	Using that $\log\tau\sim \tau-1$ as $\tau\rightarrow1$ and $\frac{\log^2\tau}{(\tau-\frac{1}{k})(\tau-1)}\sim\frac{\log^2\tau}{\tau^2}\leq\frac{c}{\tau^{2-\epsilon}}$ as $\tau\rightarrow\infty$, for every $\epsilon>0$, then the above integral is convergence for $k$ sufficiently large. This implies that
	\begin{equation}\label{a-3}
	J^2_k=o(1)\quad\quad\text{as}~~k\rightarrow\infty.
	\end{equation}
	Combining \eqref{a-1} and \eqref{a-2}, and by using \eqref{a-3}, we find that
	\begin{align}\label{a-4}
	\nonumber J_k&=2\Bigg(-\log\Big(1-\frac{1}{k}\Big)+\frac{1}{\log^2k}\int_{1}^{\infty}\frac{\frac{1}{k}-1}{(\tau-\frac{1}{k})(\tau-1)}\log^2\tau\ d\tau\Bigg)\\
	&\rightarrow0\quad\quad\text{as}~~k\rightarrow\infty.
	\end{align}\\
	\textbf{Estimate of $I_k$.}
	We have
	\begin{align*}
	I_k=&\Bigg(\int_{0}^{1/k^2}\int_{0}^{2/k^2}\cdots+\int_{0}^{1/k^2}\int_{2/k^2}^{1/k}\cdots\\
	&~~~~~~~+\int_{1/k^2}^{1/k}\int_{0}^{2/k^2}\cdots+\int_{1/k^2}^{1/k}\int_{2/k^2}^{1/k}\cdots\Bigg)\frac{(\chi_k(x)-\chi_k(y))^2}{(x-y)^2}\ dxdy\\
	&=I^1_k+I^2_k+I^3_k
	\end{align*}
	where
	\begin{equation*}
	I^1_k:=\int_{0}^{1/k^2}\int_{0}^{2/k^2}\frac{(\chi_k(x)-\chi_k(y))^2}{(x-y)^2}\ dxdy,\quad I^2_k:=\int_{1/k^2}^{1/k}\int_{2/k^2}^{1/k}\frac{(\chi_k(x)-\chi_k(y))^2}{(x-y)^2}\ dxdy
	\end{equation*}
	and
	\begin{equation*}
	I^3_k:=\Bigg(\int_{0}^{1/k^2}\int_{2/k^2}^{1/k}\cdots+\int_{1/k^2}^{1/k}\int_{0}^{2/k^2}\cdots\Bigg)\frac{(\chi_k(x)-\chi_k(y))^2}{(x-y)^2}\ dxdy.
	\end{equation*}
	It now suffices to estimate $I^1_k, I^2_k$ and $I^3_k$.\\
	Concerning $I^1_k$, we have
	\begin{align}\label{a-5}
	\nonumber I^1_k&=\int_{0}^{1/k^2}\int_{1/k^2}^{2/k^2}\frac{\chi_k(x)^2}{(x-y)^2}\ dxdy=\frac{1}{\log^2k}\int_{0}^{1/k^2}\int_{1/k^2}^{2/k^2}\frac{(\log k^2x)^2}{(x-y)^2}\ dxdy\\
	\nonumber&\stackrel{\substack{\tau=k^2x\\ t=k^2y}}{=}\frac{1}{\log^2k}\int_{0}^{1}\int_{1}^{2}\frac{\log^2 \tau}{(\tau-t)^2}\ d\tau dt=\frac{1}{\log^2k}\int_{0}^{1}\int_{1}^{2}\frac{(\log \tau-\log1)^2}{(\tau-t)^2}\ d\tau dt\\
	\nonumber&~~~\leq\frac{c}{\log^2k}\int_{0}^{1}\int_{1}^{2}\frac{(\tau-1)^2}{(\tau-t)^2}\ d\tau dt=\frac{c}{\log^2k}\int_{1}^{2}\int_{0}^{1}\frac{(\tau-1)^2}{(\tau-t)^2}\  dtd\tau\\
	&~~~=\frac{c}{\log^2k}\int_{1}^{2}(\tau-1)^2\Big(\frac{1}{\tau-1}-\frac{1}{\tau}\Big)=\frac{c'}{\log^2k}.
	\end{align} 
	Next, as regards $I^2_k$, the change of variables $\tau=k^2x$ and $t=k^2y$ gives
	\begin{align}\label{a-6}
	\nonumber I^2_k&=\int_{1/k^2}^{1/k}\int_{2/k^2}^{1/k}\frac{(\log k^2x-\log k^2y)^2}{(x-y)^2}\ dxdy=\frac{1}{\log^2k}\int_{1}^{k}\int_{2}^{k}\frac{(\log\tau-\log t)^2}{(\tau-t)^2}\ d\tau dt\\
	\nonumber&=\frac{1}{\log^2k}\int_{1}^{k}\int_{2}^{k}\frac{(\log(\tau/t))^2}{(\tau-t)^2}\ d\tau dt\stackrel{r=\tau/t}{=}\frac{1}{\log^2k}\int_{1}^{k}\frac{1}{t}\int_{2/t}^{k/t}\frac{\log^2r}{(r-1)^2}\ drdt\\
	&\leq\frac{1}{\log^2k}\int_{1}^{k}\frac{dt}{t}\int_{0}^{\infty}\frac{\log^2r}{(r-1)^2}\ dr=\frac{c}{\log k}.
	\end{align} 
	For $I^3_k$, we have
	\begin{align*}
	I^3_k&\leq 2\int_{0}^{2/k^2}\int_{1/k^2}^{1/k}\frac{(\chi_k(x)-\chi_k(y))^2}{(x-y)^2}\ dxdy=2\int_{0}^{2/k^2}\int_{1/k^2}^{1/k}\frac{(\chi_k(x)-\chi_k(y))^2}{(x-y)^2}\ dxdy\\
	&=2\int_{0}^{1/k^2}\int_{1/k^2}^{1/k}\frac{(\chi_k(x)-\chi_k(y))^2}{(x-y)^2}\ dxdy+2\int_{1/k^2}^{2/k^2}\int_{1/k^2}^{1/k}\frac{(\chi_k(x)-\chi_k(y))^2}{(x-y)^2}\ dxdy.
	\end{align*}
	Now,
	\begin{align}\label{a-7}
	\nonumber&\int_{0}^{1/k^2}\int_{1/k^2}^{1/k}\frac{(\chi_k(x)-\chi_k(y))^2}{(x-y)^2}\ dxdy=\frac{1}{\log^2k}\int_{0}^{1/k^2}\int_{1/k^2}^{1/k}\frac{(\log k^2x)^2}{(x-y)^2}\ dxdy\\
	\nonumber&\stackrel{\substack{\tau=k^2x\\ t=k^2y}}{=}\frac{1}{\log^2k}\int_{0}^{1}\int_{1}^{k}\frac{\log^2\tau}{(\tau-t)^2}\ d\tau dt=\frac{1}{\log^2k}\int_{1}^{k}\Big(\frac{1}{(\tau-1)^2}-\frac{1}{\tau^2}\Big)\log^2\tau\ d\tau\\
	&\leq\frac{1}{\log^2k}\int_{1}^{\infty}\Big(\frac{1}{(\tau-1)^2}-\frac{1}{\tau^2}\Big)\log^2\tau\ d\tau=\frac{c}{\log^2k}.
	\end{align}
	Arguing as in the case of $I^2_k$, we have that
	\begin{align}\label{a-8}
	\nonumber&\int_{1/k^2}^{2/k^2}\int_{1/k^2}^{1/k}\frac{(\chi_k(x)-\chi_k(y))^2}{(x-y)^2}\ dxdy=\frac{1}{\log^2k}\int_{1}^{k}\int_{1}^{2}\frac{(\log t-\log\tau)^2}{(t-\tau)^2}\ dtd\tau\\
	\nonumber&\stackrel{r=t/\tau}{=}\frac{1}{\log^2k}\int_{1}^{k}\frac{d\tau}{\tau}\int_{1/\tau}^{2/\tau}\frac{\log^2r}{(r-1)^2}\ dr\leq\frac{1}{\log^2k}\int_{1}^{k}\frac{d\tau}{\tau}\int_{1}^{\infty}\frac{\log^2r}{(r-1)^2}\ dr\\
	&~~=\frac{c}{\log k}.
	\end{align}
	Putting together \eqref{a-5}, \eqref{a-6}, \eqref{a-7} and \eqref{a-8}, we find that
	\begin{equation}\label{a-9}
	I_k\leq\frac{c}{\log^2k}+\frac{c}{\log k}\rightarrow0\quad\quad\text{as}~~k\rightarrow\infty.
	\end{equation}
	From \eqref{a-4} and \eqref{a-9}, we conclude that
	\begin{equation}\label{a-26}
	[\chi_k-1]_{H^{1/2}(\R_+)}\rightarrow0\quad\quad\text{as}~~ k\rightarrow\infty.
	\end{equation}
	Now, \eqref{H-half-convergence} follows by combining \eqref{L2-convergence} and \eqref{a-26}. As wanted.
\end{proof}
\begin{defi}\label{def:Lipschitz}
We say that an open subset $\O$ of $\R^N$ is  Lipschitz  if for each $q\in \partial\Omega$, there exist  a tangent hyperplane $H_q$,  a normal $N_q$ of $H_q$, $r_q>0$, open $r_q$-balls $B_{r_q}\subset H_q$ and a function $\Phi_q:  B_{r_q}\times I\to \R^N$ such that
\begin{itemize}
	\item [$(i)$] $\Phi_q(B_{r_q}\cap H^+_q)\subset\Omega$
	\item[$(ii)$] $\Phi_q(B_{r_q}\cap \partial H^+_q)\subset\partial\Omega$
	\item[$(iii)$] $C^{-1}|x-y|\leq|\Phi_q(x)-\Phi_q(y)|\leq C|x-y|,~~~C>1,~~x,y\in B_{r_q}\times I,~~I\subset\R$.
\end{itemize}
Here, $H^+_q$ is the upper half-tangent hyperplane containing $N_q$.  Put $Q_q:=B_{r_q}\times (-r_q,r_q)$ and we recall that $B_{r_q}$ is a $(N-1)$-ball.\\
\end{defi}

\begin{remark}\label{rem:Lipschitz}
	We would like to make the following observation. It is well-known that a domain $\Omega$ is said to be \textit{strongly} Lipschitz if its boundary can be seen as a local graph of a Lipschitz function $\phi:\R^{N-1}\rightarrow\R$. Moreover, by mean of a vectorfield $\eta$ (with $|\eta|=1$ on $\partial\Omega$) which is globally transversal \footnote{$\eta$ is said to be globally tranversal to $\partial\Omega$ if there is $\kappa>0$ such that $\eta\cdot\nu\geq\kappa$ a.e. on $\partial\Omega$. Here $\nu$ is the unit normal vector to $\partial\Omega$.} to $\partial\Omega$, one can construct a bi-Lipschitz mapping via $\phi$. In particular, $\Omega$ fulfills properties $(i)$-$(iii)$. However, every Lipschitz domain in the sense of definition $(i)$-$(iii)$ is not necessarily a local graph of a Lipschitz function. This clearly shows that \textit{strongly} Lipschitz domain is also a Lipschitz domain. But the converse is not true. This is consistent with the fact that \textit{strongly} Lipschitz domains are not stable under bi-Lipschitz map. See \cite{hofmann2007geometric} for more details.
\end{remark}

Clearly, there exists $\beta>0$ such that 
\begin{equation}
\overline{\Omega_{\beta}}:=\{0\leq\delta_{\Omega}(x)\leq\beta\}\subset\cup_{q\in\partial\Omega}\Phi_q(Q_{q}).
\end{equation}
We recall that $\Omega_{\beta}$ is the so-called \textit{inner tubular neighbourhood} of $\Omega$. By compactness, there exists $m\in \N$ such that 
\begin{equation}
\overline{\Omega_{\beta}}:=\{0\leq\delta_{\Omega}(x)\leq\beta\}\subset\cup^{m}_{j=1}\Phi_{q_j}(Q_{q_j}).
\end{equation}
We will write $j$ in the place of $q_j$ provided there is no ambiguity.
For $j=1,\dots,m$, let $u_k^j$ be a sequence define by
\begin{equation*}
u_k^j(\Phi_{j}(x))=\chi_k(x_N),~~~\forall x\in Q_j,
\end{equation*}
where $\chi_k$ is defined  in \eqref{test-function-1d}.  Equivalently, $u_k^j$ can be defined as
\begin{equation}\label{b-8}
u_k^j(x)=\chi_k(\Phi^{-1}_{j}(x)\cdot N_{j}),~~~\forall x\in\Omega.
\end{equation}
Define $\cO_j:=\Phi_j(Q_j)$ and $\cO_{m+1}=\Omega\setminus\overline{\Omega_{\beta}}$. We also write $Q_j^+:=B_{r_j}\times (0,r_j)$.\\

We have the following.
\begin{lemma}\label{an-intermediate-estimate}
	For all $j=1,\dots,m$ there exists a positive constant $C>0$ depending only on $j, m, \Omega$ and $N$ such that
	\begin{equation}\label{eq:first-est}
	\|u^j_k-\mathbbm{1}_{\Omega}\|_{H^{1/2}(\cO_j\cap\Omega)}\leq C\|\chi_k-1\|_{H^{1/2}(0,r_j)}.
	\end{equation} 
\end{lemma}
\begin{proof}
	For $j=1,\dots,m$, by using the change of variables $x=\Phi_j(z)$ and $y=\Phi_j(\overline{z})$, we get
	\begin{align}
	&\int_{\cO_j\cap\Omega}\int_{\cO_j\cap\Omega}\frac{(u_k(x)-u_k(y))^2}{|x-y|^{N+1}}\ dxdy=\int_{Q_j^+}\int_{Q_j^+}\frac{(u_k(\Phi_j(z))-u_k(\Phi_j(\overline{z})))^2}{|\Phi_j(z)-\Phi_j(\overline{z})|^{N+1}}\ dzd\overline{z} \nonumber\\
	&=\int_{Q_j^+}\int_{Q_j^+}\frac{(\chi_k(z_N)-\chi_k(\overline{z}_N))^2}{|\Phi_j(z)-\Phi_j(\overline{z})|^{N+1}}\ dzd\overline{z}\leq C\int_{Q_j^+}\int_{Q_j^+}\frac{(\chi_k(z_N)-\chi_k(\overline{z}_N))^2}{|z-\overline{z}|^{N+1}}\ dzd\overline{z} \nonumber\\
	& \leq C\int_{B_{r_j}}\int_{B_{r_j}}\int_{0}^{r_j}\int_{0}^{r_j}\frac{(\chi_k(z_N)-\chi_k(\overline{z}_N))^2}{|z-\overline{z}|^{N+1}}\ dzd\overline{z} \nonumber\\
	&\leq C\int_{B_{r_j}}dz'\int_{H_j}d\ov z'\int_{0}^{r_j}\int_{0}^{r_j}\frac{(\chi_k(z_N)-\chi_k(\overline{z}_N))^2}{(|z'-\overline{z}'|^2+|z_N-\overline{z}_N|^2)^{\frac{N+1}{2}}}\ dz_Nd\overline{z}_N .
	\label{b-16}
	\end{align}
	By translation and rotation, we  have 
	\begin{align*}
	&\int_{B_{r_j}}dz'\int_{H_j}d\ov z'\int_{0}^{r_j}\int_{0}^{r_j}\frac{(\chi_k(z_N)-\chi_k(\overline{z}_N))^2}{(|z'-\overline{z}'|^2+|z_N-\overline{z}_N|^2)^{\frac{N+1}{2}}}\ dz_Nd\overline{z}_N \\
	&= \int_{ B_{r_j}} dz'\int_{\R^{N-1}} d\ov z'\int_{0}^{r_j}\int_{0}^{r_j}\frac{(\chi_k(z_N)-\chi_k(\overline{z}_N))^2}{(|z'-\overline{z}'|^2+|z_N-\overline{z}_N|^2)^{\frac{N+1}{2}}}\ dz_Nd\overline{z}_N \\
	&\leq CA\int_{0}^{r_j}\int_{0}^{r_j}\frac{(\chi_k(z_N)-\chi_k(\overline{z}_N))^2}{ |z_N-\overline{z}_N|^2 }\ dz_Nd\overline{z}_N ,
	\end{align*}
	where $A= \int_{\R^{N-1}}\frac{dl}{(1+|l|^2)^{(N+1)/2}}\leq C$ and $B_{r_j}$ is a bounded  open subset of $\R^{N-1}$.
	Therefore, since the estimate of the $L^2$ norm follows easily,  this and \eqref{b-16} give \eqref{eq:first-est}, concluding the proof.
\end{proof}
Consider  $0\leq \psi_j\in C^\infty_c(\cO_j)$ a partitioning of unity subordinated to $\{\cO_j\}_{j=1,\dots,m+1}$. Define
\begin{equation}\label{approximation-sequence}
u_k:=\sum_{j=1}^{m+1}\psi_ju_k^j \in C^{0,1}_c(\Omega),
\end{equation} 
where   $u^{m+1}_k\equiv1$ on $\Omega$.
We have the following approximation.
\begin{lemma}\label{approximation-result}
	There holds
	\begin{equation}
	\|u_k-\mathbbm{1}_{\Omega}\|_{H^{1/2}(\Omega)}\rightarrow0\quad\quad\text{as}~~k\rightarrow\infty.
	\end{equation}
\end{lemma}
\begin{proof}
	We estimate
	\begin{align*}
	[u_k-\mathbbm{1}_\Omega]_{H^{1/2}(\Omega)}^2&\leq \left( \sum_{j=1}^{m+1} [\psi_j u_k^j-\psi_j]_{H^{1/2}(\Omega)}\right)^2 \leq m \sum_{j=1}^m [\psi_j u_k^j-\psi_j]_{H^{1/2}(\Omega)}^2\\
	&\leq  C \sum_{j=1}^{m} \int_{\cO_j\cap\Omega\times \cO_j\cap\Omega}\dots\,dxdy+    C\sum_{j=1}^{m} \int_{ \Omega\setminus \cO_j\times\Omega\cap \cO_j }\dots\, dxdy\\
	&=: C I_1(k)+ C I_2(k).
	\end{align*}
	We now estimate $I_1(k)$ and $I_2(k)$. Let us start with $I_2(k)$.\\
	
	We have
	\begin{align}\label{eq:e0}
	\nonumber I_2(k)=& \sum_{j=1}^{m} \int_{\Omega\setminus \cO_j\times \Omega\cap \cO_j}\frac{[(\psi_j u_k^j-\psi_j)(x)-(\psi_j u_k^j-\psi_j)(y)]^2}{|x-y|^{N+1}}dxdy\\
	\nonumber &= \sum_{j=1}^{m} \int_{ \Omega \setminus \cO_j}\frac{dx}{|x-y|^{N+1}} \int_{  \Omega\cap  \textrm{Supp}\psi_j}{ (\psi_j u_k^j-\psi_j)(y)^2}dy\\
	\nonumber&\leq  C \sum_{j=1}^{m} \textrm{dist}( \textrm{Supp}\psi_j, \partial\cO_j)^{-N-1}  \int_{   \Omega\cap \cO_j  }\psi_j^2 | u_k^j(y)-1|^2dy\\
	&\leq C(N)\max_{1\leq j\leq {m}} \textrm{dist}( \textrm{Supp}\psi_j, \partial\cO_j)^{-N-1}   \sum_{j=1}^{m}  \|u_k^j-\mathbbm{1}_{\Omega}\|_{L^{2}( \Omega\cap \cO_j  )}^2.
	\end{align}
	Now regarding $I_1(k)$, we have 
	\begin{align*}
	&I_1(k)=\sum_{j=1}^{m}\int_{\cO_j\cap\Omega}\int_{\cO_j\cap\Omega}\frac{[\psi_j(x)(u^j_k(x)-1)-\psi_j(y)(u^j_k(y)-1)]^2}{|x-y|^{N+1}}\ dxdy\\
	&=\sum_{j=1}^{m}\int_{\cO_j\cap\Omega}\int_{\cO_j\cap\Omega}\frac{[\psi_j(x)((u^j_k(x)-1)-(u^j_k(y)-1))+(\psi_j(x)-\psi_j(y))(u^j_k(y)-1)]^2}{|x-y|^{N+1}}\ dxdy\\
	&\leq2\sum_{j=1}^{m}\int_{\cO_j\cap\Omega}\int_{\cO_j\cap\Omega}\frac{\psi_j(x)^2[(u^j_k(x)-1)-(u^j_k(y)-1)]^2}{|x-y|^{N+1}}\ dxdy\\
	&~~~~~~~~~+2\sum_{j=1}^{m}\int_{\cO_j\cap\Omega}\int_{\cO_j\cap\Omega}\frac{(\psi_j(x)-\psi_j(y))^2(u^j_k(y)-1)^2}{|x-y|^{N+1}}\ dxdy\\
	&=I^1_1(k)+I^2_1(k),
	\end{align*}
	where
	\begin{align}\label{eq:e1}
	\nonumber I^1_1(k)&=2\sum_{j=1}^{m}\int_{\cO_j\cap\Omega}\int_{\cO_j\cap\Omega}\frac{\psi_j(x)^2[(u^j_k(x)-1)-(u^j_k(y)-1)]^2}{|x-y|^{N+1}}\ dxdy\\
	\nonumber&\leq2\sum_{j=1}^{m}\int_{\cO_j\cap\Omega}\int_{\cO_j\cap\Omega}\frac{[(u^j_k(x)-1)-(u^j_k(y)-1)]^2}{|x-y|^{N+1}}\ dxdy\qquad\quad(\text{since}~0\leq\psi_j\leq1)\\
	&=c\sum_{j=1}^{m}[u^j-\mathbbm{1}_{\Omega}]^2_{H^{1/2}(\cO_j\cap\Omega)} 
	\end{align}
	and
	\begin{equation*}
	I^2_1(k)=2\sum_{j=1}^{m}\int_{\cO_j\cap\Omega}\int_{\cO_j\cap\Omega}\frac{(\psi_j(x)-\psi_j(y))^2(u^j_k(y)-1)^2}{|x-y|^{N+1}}\ dxdy.
	\end{equation*}
	Using that $\psi_j$ is Lipschitz, we get
	\begin{align*}
	&2\int_{\cO_j\cap\Omega}\int_{\cO_j\cap\Omega}\frac{(\psi_j(x)-\psi_j(y))^2(u^j_k(y)-1)^2}{|x-y|^{N+1}}\ dxdy\\
	&\leq c(j)^2\iint_{|x-y|<1}\frac{(u^j_k(y)-1)^2|x-y|^2}{|x-y|^{N+1}}\ dxdy+8\iint_{|x-y|\geq1}\frac{(u^j_k(y)-1)^2}{|x-y|^{N+1}}\ dxdy\\
	&\leq
	\tilde{c}(j)\|u^j_k-\mathbbm{1}_{\Omega}\|^2_{L^2(\cO_j\cap\Omega)}
	\end{align*}
	which implies that
	\begin{equation}\label{eq:e2}
	I^2_1(k)\leq\max_{1\leq j\leq m}\tilde{c}(j)\sum_{j=1}^{m}\|u^j_k-\mathbbm{1}_{\Omega}\|^2_{L^2(\cO_j\cap\Omega)}.
	\end{equation}
	Finally, \eqref{eq:e0}, \eqref{eq:e1} and \eqref{eq:e2} yield 
	\begin{align}\label{eq:claim-1}
	\nonumber\|u_k-\mathbbm{1}_\Omega\|_{H^{1/2}(\Omega)}^2&=\|u_k-\mathbbm{1}_\Omega\|^2_{L^2(\Omega)}+[u_k-\mathbbm{1}_{\Omega}]^2_{H^{1/2}(\Omega)}\\
	\nonumber&\leq c\sum_{j=1}^{m}\|u^j_k-\mathbbm{1}_{\Omega}\|^2_{L^2(\cO_j\cap\Omega)}+C I_1(k)+C I_2(k)\\
	&=\tilde{c}\sum_{j=1}^{m}\|u^j_k-\mathbbm{1}_{\Omega}\|^2_{H^{1/2}(\cO_j\cap\Omega)} \leq    C(N,m) \sum_{j=1}^{m} \|\chi_k-1\|_{H^{1/2}(0,r_j)}^2.
	\end{align}
	In the latter inequality, we used Lemma \ref{an-intermediate-estimate}. Now, since from Lemma \ref{1-d-approximation} there holds $\|\chi_k-1\|_{H^{1/2}(0,r_j)}^2\to 0$ as $k\to \infty$, we complete the proof by letting $k\rightarrow\infty$ in the inequality \eqref{eq:claim-1}.
\end{proof}

As a direct consequence of the above approximation results, we have the following.
\begin{prop}\label{asymptotice-of-sobolev-constant-for-s-geq-1/2}
	Let $N\geq2,~s\in(0,1/2]$ and let $\Omega\subset\R^N$ be a bounded Lipschitz domain. Then
	\begin{equation}
	S_{N,s}(\Omega)=0.
	\end{equation}
\end{prop}
Before proving the proposition above, we mention that our result extends to $s=1/2$ the one obtained in \cite[Lemma 16]{frank2018minimizers}. Below, we give the 
\begin{proof}[Proof of Proposition \ref{asymptotice-of-sobolev-constant-for-s-geq-1/2}]
	By definition
	\begin{equation}\label{alternative-minimization-problem-on-domain-for-s=1/2}
	S_{N,s}(\Omega)=\inf_{\substack{u\in H^{s}_0(\Omega)\\ u\neq0}}\frac{Q_{N,s,\Omega}(u)}{\|u\|^2_{L^{2^*_s}(\Omega)}}=\inf_{\substack{u\in C^{0,1}_c(\Omega)\\ u\neq0}}\frac{Q_{N,s,\Omega}(u)}{\|u\|^2_{L^{2^*_s}(\Omega)}},
	\end{equation}
	where $C^{0,1}_c(\O)$ is the space of Lipschitz  functions with compact support.  Now by  Lemma \ref{approximation-result}, we get 
	\begin{equation}\label{ineq-2}
	0\leq S_{N,s}(\Omega)\leq \frac{Q_{N,s,\Omega}(u_k)}{\|u_k\|^2_{L^{2^*_s}(\Omega)}}\leq C(N,s) \frac{Q_{N,1/2,\Omega}(u_k)}{\|u_k\|^2_{L^{2^*_s}(\Omega)}}=C(N,s)\frac{[u_k-\mathbbm{1}_{\Omega}]_{H^{1/2}(\Omega)}}{\|u_k\|^2_{L^{2^*_s}(\Omega)}}\to0,
	\end{equation}
where $u_k$ is defined by \eqref{approximation-sequence}, which satisfies  $\liminf_{k\to \infty}\|u_k\|^2_{L^{2^*_s}(\Omega)}>0$.
\end{proof}


\begin{thebibliography}{10}




\bibitem{barrios2015critical} B. Barrios, E. Colorado, R. Servadei and F. Soria, \emph{A critical fractional equation with concave-convex power nonlinearities.} Annales de l'Institut Henri Poincare (C) Non Linear Analysis. Vol. 32. No. 4. Elsevier Masson, 2015.

\bibitem{brezis1983relation} H. Br\'{e}zis and E. H. Lieb, \emph{A relation between pointwise convergence of functions and convergence of functionals.} Proceedings of the American Mathematical Society 88.3 (1983): 486-490.

\bibitem{chen2018dirichlet} H. Chen, \emph{The Dirichlet elliptic problem involving regional fractional Laplacian.} Journal of Mathematical Physics 59.7 (2018): 071504.

\bibitem{del2015first} L. M. Del Pezzo and A. M. Salort, \emph{The first non-zero Neumann p-fractional eigenvalue.} Nonlinear Analysis: Theory, Methods \& Applications 118 (2015): 130-143.

\bibitem{dinh2019existence} V. D. Dinh, \emph{Existence, non-existence and blow-up behavior of minimizers for the mass-critical fractional nonlinear Schr$\ddot{\text{o}}$dinger equations with periodic potentials.} arXiv preprint arXiv:1912.08750 (2019).

\bibitem{di2012hitchhiker's} E. Di Nezza, G. Palatucci and E. Valdinoci. \emph{Hitchhiker's guide to the fractional Sobolev spaces.} Bulletin des sciences mathematiques 5.136 (2012): 521-573.

\bibitem{dyda2004fractional} B. Dyda, \emph{A fractional order Hardy inequality.} Illinois Journal of Mathematics 48.2 (2004): 575-588.

\bibitem{dyda2011fractional} B. Dyda and R. L. Frank, \emph{Fractional Hardy-Sobolev-Maz'ya inequality for domains.} Studia Mathematica 2.208 (2012): 151-166.

\bibitem{fall2020regional} M. M. Fall, \emph{Regional fractional Laplacians: Boundary regularity.} arXiv preprint \href{https://arxiv.org/abs/2007.04808v1}{https://arxiv.org/abs/2007.04808v1} (2020).

\bibitem{fall2012nonexistence} M. M. Fall and T. Weth, \emph{Nonexistence results for a class of fractional elliptic boundary value problems.} Journal of Functional Analysis 263.8 (2012): 2205-2227.

\bibitem{frank2018minimizers} R. L. Frank, T. Jin and J. Xiong. \emph{Minimizers for the fractional Sobolev inequality on domains.} Calculus of Variations and Partial Differential Equations 57.2 (2018): 43.

\bibitem{ghoussoub2017hardy} N. Ghoussoub, and F. Robert, \emph{The Hardy--Schr\"{o}dinger operator with interior singularity: the remaining cases.} Calculus of Variations and Partial Differential Equations 56.5 (2017): 1-54.

\bibitem{hofmann2007geometric} S. Hofmann, M. Mitrea, and M. Taylor, \emph{Geometric and transformational properties of Lipschitz domains, Semmes-Kenig-Toro domains, and other classes of finite perimeter domains.} The Journal of Geometric Analysis 17.4 (2007): 593-647.

\bibitem{lieb2002sharp} E. H. Lieb, \emph{Sharp constants in the Hardy-Littlewood-Sobolev and related inequalities.} Annals of Mathematics 118.2 (1983): 349-374.

\bibitem{ros2014pohozaev} X. Ros-Oton and J. Serra, \emph{The Pohozaev identity for the fractional Laplacian.} Archive for Rational Mechanics and Analysis 213.2 (2014): 587-628.

\bibitem{ros2017pohozaev} X. Ros-Oton, J. Serra and E. Valdinoci, \emph{Pohozaev identities for anisotropic integrodifferential operators.} Communications in Partial Differential Equations 42.8 (2017): 1290-1321.

\bibitem{schoen1984conformal} R. Schoen, \emph{Conformal deformation of a Riemannian metric to constant scalar curvature.} Journal of Differential Geometry 20.2 (1984): 479-495.



\end{thebibliography}
\bibliographystyle{ieeetr}

\end{document}